\documentclass[11pt,reqno]{amsart}

\usepackage{fullpage}

\usepackage[margin=1in]{geometry}
\usepackage{amsmath}
\usepackage{amsfonts}
\usepackage{amssymb}
\usepackage{amsthm,mathtools,accents,}
\usepackage{newlfont}
\usepackage{graphicx}
\usepackage{xcolor,enumerate,mathrsfs, scalerel, upgreek}
\newtheorem{thm}{Theorem}[section]

\newtheorem{lem}[thm]{Lemma}

\newtheorem{que}{Question}

\newtheorem{thmx}{Theorem}

\theoremstyle{definition}
\newtheorem{defn}[thm]{Definition}
\newtheorem{rem}[thm]{Remark}
\numberwithin{equation}{section}

\newcommand{\norm}[1]{\Vert#1\Vert}
\newcommand{\Norm}[1]{\left\Vert#1\right\Vert}
\newcommand{\na}{\nabla}
\newcommand{\pa}{\partial}
\newcommand{\lec}{\lesssim}
\newcommand{\td}{\tilde}

\renewcommand{\div}{\operatorname{div}}
\newcommand{\curl}{\operatorname{curl}}
\newcommand{\dist}{\operatorname{dist}}

\newcommand\al{\alpha}
\newcommand\be{\beta}
\newcommand\de{\delta}
\newcommand\De{\Delta}
\newcommand{\ga}{\gamma}

\newcommand\ep{\epsilon}

\newcommand\e {\varepsilon}
\newcommand\ka{\kappa}
\newcommand{\la}{\lambda}

\newcommand\ph{\varphi}

\renewcommand{\th}{\theta}

\newcommand{\om}{\omega}
\newcommand{\Om}{\Omega}
\newcommand{\si}{\sigma}

\newcommand{\R}{\mathbb{R}}

\newcommand{\N}{\mathbb{N}}

\newcommand{\supp}{\operatorname{supp}}

\newcommand{\I}{\textrm{Id}}
\newcommand{\loc}{{\text{loc}}}
\newcommand{\uloc}{\text{uloc}}

\renewcommand{\d}{\text{d}}

\newcommand{\ac}[1]{\accentset{\circ}#1}

\usepackage[usestackEOL]{stackengine}
\def\dint{\,\ThisStyle{\ensurestackMath{%
  \stackinset{c}{.2\LMpt}{c}{.5\LMpt}{\SavedStyle-}{\SavedStyle\phantom{\int}}}%
  \setbox0=\hbox{$\SavedStyle\int\,$}\kern-\wd0}\int}
  
\usepackage[usestackEOL]{stackengine}
\def\diint{\,\ThisStyle{\ensurestackMath{%
  \stackinset{c}{.2\LMpt}{c}{.5\LMpt}{\SavedStyle$-----$}{\SavedStyle\phantom{\iint}}}%
  \setbox0=\hbox{$\SavedStyle\iint\,$}\kern-\wd0}\iint}  
  
\title{The role of the pressure in the regularity theory\\ for the Navier-Stokes equations}
\author{Hyunju Kwon}
\thanks{School of Mathematics, Institute for Advanced Study, Princeton, NJ 08540, USA., hkwon@ias.edu.}

\begin{document}
\maketitle
\begin{abstract}
We first show the equivalence of two classes of generalized suitable weak solutions to the 3D incompressible Navier-Stokes equations allowing distributional pressure, the class of dissipative weak solutions and local suitable weak solutions. Then, an $\varepsilon$-regularity criterion for dissipative weak solutions follows from that for local suitable weak solutions. We relax the $\varepsilon$-regularity criterion with a new approach using a local version of Leray projection operator. As an application of the approach, we obtain the short-time regularity result on a bounded domain for dissipative solutions.
\end{abstract}

\section{Introduction}
Consider the three-dimensional incompressible Navier-Stokes equations 
\begin{align}\label{NS}\tag{NS}
\begin{cases}
\pa_t u + (u\cdot \na ) u + \na p = \De u \\
\div u =0
\end{cases}
\end{align}
on $(0,T)\times \R^3$. The equations describe the flow of incompressible viscous fluids and have two unknown functions $u:(0,T)\times \R^3 \to \R^3$ and $p:(0,T)\times \R^3 \to \R$, which represent the velocity of the fluid and pressure, respectively. This equations have an invariance under the scaling
\begin{align*}
u_\la(t,x):= \la u(\la^2 t, \la x), \quad
p_\la(t.x):= \la^2 p(\la^2 t, \la x), \quad \forall \la>0,
\end{align*}
which plays an important role in the regularity theory.

For any given initial data with finite kinetic energy, the existence of a weak solution to the Navier-Stokes equations for all time dates back to Leray \cite{Leray} (See also Hopf \cite{Hopf} for smooth bounded domain cases). The constructed solutions  satisfy the energy inequality:
\begin{align}\label{EI}
\int_{\R^3} |u(t,x)|^2 \d x + \int_0^t \int_{\R^3} |\na u(\tau, x)|^2 \d x \d \tau \leq \int_{\R^3} |u(0,x)|^2 \d x , \quad \forall t \geq 0,
\end{align}
so a weak solution satisfying the energy inequality is called a {\it Leray-Hopf} weak solution. The uniqueness and regularity of Leray-Hopf weak solutions remain open. Even for smooth initial data with sufficiently fast decay at spatial infinity, whether the solution stays smooth for all time is an outstanding open problem.  

To better understand the regularity, conditional regularity of a {\it weak solution} has been widely investigated, which tells us what is happening near a {\it singular} point. One of the classical questions in the regularity theory is about the interior regularity (see for example \cite{SeSv18}):
\begin{que}[Interior regularity] For the Navier-Stokes equations on $Q_1$, what are minimal conditions on a weak solution to make the origin regular?
\end{que}
\noindent Here, we denote a parabolic cylinder by $Q_r:=(-r^2,0)\times B_r\subset \R\times \R^3$, where $B_r$ is the Euclidean ball centred at the origin with the radius $r>0$. Due to the translation and scale invariance of the Navier-Stokes equations, one can work on $Q_1$ without loss of generality. A regular point is defined in the sense of Caffarelli-Kohn-Nirenberg \cite{CKN}: the origin is regular if $v\in L^\infty(Q_r)$ for some $r>0$. If a point is not regular, we call it as a singular point. 

There are many important results regarding to the interior regularty. For weak solutions, Serrin \cite{Serrin62} showed the local $L^\infty$-boundedness under the assumption $u\in L^r_tL^m_x(Q_1)$ for $(r,m)$ in the subcritical range $\frac2r +\frac3m < 1$. After that, Struwe \cite{Str88} and Takahashi \cite{Tak90} extended the result to the critical range, the equality case, when $m\in (3,\infty]$.
These results work even with distributional pressure, but it is not known whether a generic weak solution satisfies the assumption. Indeed, by interpolation, a weak solution satisfies $u\in L^r_tL^m_x(Q_1)$ for $(r,m)$ in the supercritical range $ \frac2r +\frac3m \ge \frac32$. To get around this issue, Scheffer \cite{Sch77} introduced a suitable weak solution: a weak solution $(u,p)$ satisfying $p\in L^\frac32(Q_1)$ and
the  local energy inequality, which is a local analogue of the energy inequality \eqref{EI}. Building upon \cite{Sch77, Sch76}, Caffarelli, Kohn, and Nirenberg \cite{CKN} obtained an $\e$-regularity regularity criterion for suitable weak solutions and, as a consequence, proved that the set of all singular points has one-dimensional parabolic Hausdorff measure zero. For different approaches, see also \cite{Lin98, LaSe99, Vasseur}. The following theorem is a version of $\e$-regularity criteria in \cite{NRS96, Lin98, LaSe99}. 

\begin{thmx}\label{thm:origin} There exist universal positive constants $\e$ and $C$ such that for any suitable weak solution $(u,p)$ to \eqref{NS} in $Q_1$ satisfying 
\begin{align*}
\norm{u}_{L^3(Q_1)} + \norm{p}_{L^\frac32(Q_1)} \le \e,
\end{align*}
we have
\begin{align*}
\norm{u}_{C^{\al}_{\text{par}}(Q_{1/2})} \le C
\end{align*}
for some $\al\in (0,1)$. 
\end{thmx}
\noindent 
In the developments of \cite{GuPh17, HWZ, DoWa21}, the smallness assumption in the classical theorem has been relaxed; in particular, in \cite{HWZ, DoWa21}, for any $(r,m)\in (1,\infty)\times (3/2,\infty)$ with $2/r+3/m<2$, if a suitable weak solution satisfies
\begin{align*}
\norm{u}_{L^r_tL^m_x(Q_1)} + \norm{p}_{L^1(Q_1)} \le \e(r,m),
\end{align*}
for some $\e(r,m)>0$, then $u$ is regular in $\overline{Q_{1/2}}$. For other types of $\e$-regularity criteria for suitable weak solutions, see also \cite{GKT07, DoWa21}. For the improvement of the partial regularity result by a logarithmic factor, see \cite{ChLe, ChYa, RWW}.

The paper particularly concerns the role of pressure in the $\e$-regularity criteria. Observe that the smallness assumption on pressure in Theorem \ref{thm:origin} cannot be removed, keeping both universal constants $\e$ and $C$; in such case, one can construct a counter-example from a class of explicit solutions, called parasitic solutions, suggested by Serrin \cite{Serrin62}, 
\begin{align}\label{para.sol}
u(t,x) = \Phi(t) \na H(x), \quad p(t,x) = -\Phi'(t) H(x) - \frac 12 |\Phi(t)\na H(x)|^2  
\end{align}
where $H$ is harmonic on $B_1$. Also, considering a bounded but discontinuous function $\Phi$ of time in these examples, one cannot expect the improvement of the time regularity of a weak solution $u$ under no integrability assumption on the pressure. The parasitic solutions $u$, on the other hand, are harmonic and hence smooth in space. Indeed, even with distributional pressure, $\e$-regularity criteria are still valid, if one allow the local $L^\infty$-bound of the velocity to depend on $\norm{u}_{L^\infty_tL^1_x(Q_1)}$. Since the definition of suitable weak solutions requires integrability of pressure, we first introduce two classes of solutions which generalize suitable weak solutions: {\it dissipative weak solutions} and {\it local suitable weak solutions.}

Arising in the study of turbulent flows at high Reynolds number, the class of dissipative weak solutions on periodic domains is first introduced by Duchon-Robert in \cite{DR99}. The following definition is a version on bounded domains in \cite{CLRM18}.
\begin{defn}[Dissipative weak solution]\label{def:diss} We call $(u,p)$ is a {\it dissipative weak solution} to \eqref{NS} on a domain $\mathcal{O} = (a,b)\times B_R$, $a<b\in \R$ and $R>0$ if 
\begin{enumerate}[(1)]
\item $u\in L^\infty_tL^2_x(\mathcal{O}) \cap L^2_t\dot{H}^1_x(\mathcal{O})$ and $p\in \mathcal{D}'(\mathcal{O})$.
\item $(u,p)$ solves \eqref{NS} in distribution sense: $u$ is divergence free in distribution sense and $(u,p)$ satisfies
\begin{align}\label{weak.NS}
\iint u \cdot (\pa_t \zeta + \De \zeta ) + u\otimes u : \na \zeta + p\cdot \na \zeta \d x \d t =0, \qquad \forall \zeta\in C_c^\infty(\mathcal{O};\R^3).
\end{align}
\item $(u,p)$ satisfies the local energy inequality in distribution sense:
\begin{align}\label{LEI0}
\iint |\na u|^2\xi \d x \d t 
\leq \iint \frac{|u|^2}2 (\pa_t \xi + \De \xi) + \frac{|u|^2}2 (u\cdot \na) \xi \d x \d t + \langle \div (pu) \rangle(\xi) 
\end{align}
for any non-negative function $\xi\in C_c^\infty(\mathcal{O};\R)$. Here, $\langle \div (pu) \rangle\in \mathcal D'(\mathcal{O})$ is defined by
\begin{align}\label{divpu}
\langle \div (pu)\rangle (\xi) :=-\lim_{\ep \to 0^+}\lim_{\al \to 0^+} \int (p\ast \psi_{\al,\ep})(u\ast \psi_{\al,\ep}\cdot \na) \xi \d x \d t ,
\end{align}
and the limit is independent of a space-time mollifier (see \cite[Proposition 1.1]{CLRM18}) of the form $\psi_{\al,\ep}(t,x) := \be_\al(t) \ga_\ep(x) $, where $\be_\al(t) := {\al}^{-1} \be(\al^{-1}t)$, $\ga_\ep(x) := \ep^{-3} \ga(\ep^{-1} x)$ for $\al>0$, $\ep>0$ and two functions $\be\in C_c^\infty(\R)$ and $\ga\in C_c^\infty(\R^3)$ with $\int_{\R} \be \,\d t = \int_{\R^3} \ga \,\d x =1$. (Note that each $\xi$ has $\supp(\xi)\Subset \mathcal{O}$, so that the integral in the right hand side of \eqref{divpu} is well-defined for sufficiently small $\al>0$ and $\ep>0$.)
\end{enumerate}
\end{defn}
\noindent Obviously, any suitable weak solution is a dissipative solution, but the converse is not true due to the Serrin's examples \eqref{para.sol}. Also, on the periodic domain, the solutions constructed by Leray are dissipative solutions \cite{DR99}. For these solutions, a partial regularity result is recently obtained by Chamorro, Lemari{\'e}-Rieusset, and Mayoufi \cite{CLRM18}. 
\medskip

The class of local suitable weak solutions, on the other hand, are introduced by Wolf \cite{Wolf15} in order to remove the assumption on pressure in $\e$-criteria for suitable weak solutions. 
\begin{defn}[local suitable weak solution]\label{defn:LSWS} We call $u$ is a {\it local suitable weak solution} to \eqref{NS} on a domain $\mathcal{O} = (a,b)\times B_R$, $a<b\in \R$ and $R>0$ if 
\begin{enumerate}[(1)]
\item $u\in L^\infty_tL^2_x(\mathcal{O})\cap L^2_t\dot{H}^1_x(\mathcal{O})$
\item $u$ solves the weak form of \eqref{NS}: $u$ is divergence-free in distribution sense and satisfies
\begin{align*}
\iint_{\R\times \R^3} u\cdot (\pa_t \zeta + \De \zeta ) + u\otimes u :\na \zeta =0, 
\end{align*}
for any divergence-free $ \zeta \in C_{c}^\infty(\mathcal{O};\R^3)$.
\item It satisfies the local energy inequality for $v=u+ \na p_{h, B_R}$ in distribution sense:
\begin{align*}
(\pa_t -\De)\frac{|v|^2}2 + |\na v|^2 + \div \left(\frac{|v|^2u}2\right) - v\cdot \div (\na p_{h, B_R} \otimes u) + \div(p_{o, B_R} v) \leq 0.
\end{align*}
\end{enumerate}
where $\na p_{h, B_R} := - E_{B_R}^*(u)$ and $\na p_{o, B_R} = - E_{B_R}^* (\div(u\otimes u)) + E_{B_R}^*(\De u)$. 
\end{defn}
\noindent Here, a bounded linear operator $E_{\Om}^*$ on $W^{-1,q}(\Om)$\footnote{The operator $E_{\Om}^*$ can be extended to the one on  $L^s(a,b;W^{1,-q}(\Om))$, $s\in [1,\infty]$, defining $E_{\Om}^*(f)(t):= E_{\Om}^*(f(t))$.}, for a bounded Lipschitz domain $\Om\subset \R^3$ and $q\in (1,\infty)$, is defined by $E_\Om^*(f) = \na \pi$ where $\pi$ is obtained as a unique solution $(w,\pi)\in W^{1,q}_{0}(\Om) \times L^q(\Om)$ to the steady Stokes system\footnote{For the theory of steady Stokes system, see \cite{GSS94, BrSh95}. For further discussion about the operator $E_{\Om}^*$, see \cite{Wolf15, ChWo17}.}
\begin{equation*}
-\De w + \na \pi = f, \quad \div w=0 \quad \text{on } \Om, \qquad w|_{\pa \Om} = 0, \quad \int_{\Om} \pi \, \d x =0.
\end{equation*}
The definition of local suitable weak solutions is motivated by the local representation of $\na p =\na ( p_{o, B_R}+ \pa_t p_{h, B_R})$ in distribution sense, whenever $(u,p)$ solves \eqref{NS} in distribution sense. We remark that these solutions include suitable weak solutions and parasitic solutions. Also, the $\e$-regularity criterion has been further improved in \cite{ChWo17, JWZ, WWZ}; in particular the smallness assumption was relaxed to 
\begin{align*}
\norm{u}_{L^m(Q_1)}\leq \e(m), \qquad \text{for } m>5/2 
\end{align*}
in \cite{WWZ}. The partial regularity result \cite{CKN} is extended to local suitable weak solutions in \cite{Wolf15} and improved by a logarithmic factor in \cite{WWZ}.  
\medskip

In this paper, we first show that these two classes of  generalized suitable weak solutions are equivalent.
\begin{thm}\label{lem:equi} If $u$ is a local suitable weak solution to \eqref{NS} on $(0,T)\times B_R$, $T,R\in(0,\infty) $, then $(u,p_{o, B_R}+ \pa_t p_{h, B_R})$ is a dissipative solution. Conversely, if $(u,p)$ is a dissipative solution, then $u$ is a local suitable weak solution. 
\end{thm}
\noindent As a consequence, the results regarding local suitable weak solutions work for dissipative solutions. In particular, an $\e$-regularity criterion with the smallness assumption on $\norm{u}_{L^m(Q_1)}$ for $m>5/2$ holds for dissipative solutions. 
In this context, we improve the result and, more importantly, provide a new proof.    
\begin{thm}\label{main.thm}
For any $r,m\in (2,\infty]$ with $2/r+ 3/m<2$, we can find $\e=\e(r,m)>0$ and a universal constant $C>0$ such that for any dissipative solution to \eqref{NS} on $Q_2$ satisfying
\begin{align}\label{small.u}
\norm{u}_{L^r_t L^m_x(Q_2)}  \leq \e,
\end{align}
we have a decomposition $u = v + h$ on $Q_1$ such that $h$ is a harmonic function on $Q_1$ with $\norm{\na^k h}_{L^\infty((-t_0,0)\times B_1)}\lec_k \norm{u}_{L^\infty(-t_0,0;L^1(B_2))}$ for any integer $k\geq 0$ and $t_0\in (0,1]$, and 
\begin{align*}
\norm{v}_{C^\al_{\text{par}}(Q_{1/2})}\leq C
\end{align*}
for some $\al\in (0,1/2-1/r)$. In particular,
\begin{align*}
\norm{u}_{L^\infty_tC^\al_x(Q_{1/2})} 
\leq C + \norm{u}_{L^\infty(-1/4,0;L^1(B_2))}
\end{align*}
\end{thm}
\begin{rem} In the case of the Serrin's examples \eqref{para.sol}, we have $v=0$ and $h=u$. 
\end{rem}
\noindent Our approach, inspired by \cite{CLRM18}, is based on taking a local version of Leray projection operator, $-\curl\De^{-1}(\ph \curl \cdot)$ for some spatial cut-off $\ph$, to \eqref{NS}. In this way, one can remove the pressure term $\na p$ from the discussion. Then, a dissipative solution can be decomposed into the principal part $v= -\curl\De^{-1}(\ph \curl u)$ and harmonic part $h=u-v$. where $v$ is a suitable weak solution to perturbed Navier-Stokes equations with smooth perturbation $h$ and some smooth external force. Theorem \ref{main.thm} then follows from the suitable modification of the proof of $\e$-regularity criteria for suitable weak solutions.
We note that in \cite{CLRM18}, the perturbation terms $(v\cdot\na )h + (h\cdot\na) v$ are considered as a part of the external force, which deteriorates the regularity of force. 

Another application of the decomposition of the dissipative solution is the short-time regularity on a bounded domain. It is well-known that on a whole domain, if an initial datum $u_0\in L^p(\R^3)$ for $p\geq 3$ then a unique (mild) solution to the Navier-Stokes equation exists for a short-time $[0,T)$, which is smooth on $(0,T)$ \cite{Kato84, Leray}. Similar results are also known in the critical Besov spaces \cite{Planchon}, $BMO^{-1}(\R^3)$ \cite{KoTa}, and $L^3_\uloc(\R^3)$ \cite{MaTe}, where 
$\norm{u_0}_{L^p_\uloc(\R^3)} := \sup_{x_0\in \R^3} \norm{u_0}_{L^p(B_1(x_0))}$. 
One then may ask whether the smoothing is a local phenomenon; Jia-\v{S}ver\'ak \cite{JS14} first raised the following question.
\begin{que} If $u_0\in L^q(B_2)$ for some $q\geq 3$, then is an associated weak solution $u$ with the initial data $u_0$ regular in $(0,t_0)\times B_1$ for some time $t_0$?
\end{que}
\noindent Analogous answers to this question is not obvious because of the non-local nature of pressure. Indeed, if the initial data outside of the ball has very low regularity, one may not be possible to expect smoothing of the solution, as in the whole domain case. 

The short-time regularity of a {\it local Leray solution} on $[0,T)\times \R^3$ was obtained in the subcritical case $u_0\in L^q(B_{2})$, $q>3$, in \cite{JS14}, and in the critical case  $u_0\in L^3(B_{2})$ in \cite{KMT20, BP20}, when it satisfies additional global information
\begin{align*}
\norm{u_0}_{L^2_\uloc(\R^3)}^2 := \sup_{x_0\in \R^3} \norm{u_0}_{L^2(B_1(x_0))}^2 <+\infty, \quad
\lim_{|x_0|\to \infty} \norm{u_0}_{L^2(B_1(x_0))}^2=0.
\end{align*}
Indeed, the definition of local Leray solutions provides the representation of the pressure and its estimate controlled by the global information of initial data, so that the difficulty arising from the pressure can be resolved and the length of time interval $t_0$ depends on $\norm{u_0}_{L^2_\uloc(\R^3)}$. The result is further extended in \cite{KMTs20, KMTs21} 
to local energy solutions with initial data in local Morrey spaces. The results in \cite{KMT20, KMTs20}, on the other hand, are corollaries of the short-time regularity of suitable weak solutions, in which case $t_0$ depends on the size of the pressure in $L^2_tL^\frac32_x$. We remark that all the above-mentioned results got around the difficulty originating from pressure either by additional assumptions on the initial data globally in space or by at least some integrability assumption on pressure. As a result, the length of time $t_0$ depends on a certain size of the pressure directly and indirectly.

The last result of the paper assert that even with distributional pressure the short-time smoothing always occurs. Indeed, the result is stated for dissipative weak solutions which includes distributional pressure, and the length of time for which the regularity of the solutions is guaranteed is independent of the size of the pressure.
\begin{thm}\label{thm:app} There exists an universal constant $\e>0$ such that if a dissipative solution $(u,p)$ to the Navier-Stokes equations \eqref{NS} in $(0,T_0)\times B_2$, for some $T_0\in(0,1)$ has an initial data $u_0$ satisfying 
\begin{align*}
\norm{u_0}_{L^3(B_2)}\leq \e
\end{align*}
in the sense of {$\lim_{t\to 0^+}\norm{u(t,\cdot)-u_0}_{L^2(B_2)}=0$,}
and satisfies $\norm{u}_{L^3((0,T_0)\times B_2)}\leq M$ for some $M\in (0,\infty)$, then it is regular in $(0,T] \times B_{1/2}$, $T=\min(T_0, C_*(1+M^{42})^{-1})$ for some positive universal constant $C_*$.
\end{thm}

This paper is organized as follows. In Section \ref{sec:dec}, we introduce our main observation, a decomposition lemma for dissipative solutions. Based on the lemma, we obtain the $\e$-regularity results, Theorem \ref{main.thm}, in Section \ref{sec:reg}-\ref{sec:proof}. The short-time regularity result, Theorem \ref{thm:app}, is established in Section \ref{sec:stlr}. Lastly, the proof of Theorem \ref{lem:equi} is given in Section \ref{sec:appenB}.

\subsection*{Notations} Given two comparable quantities $X$ and $Y$, the inequality $X\lec Y$ means $X\leq CY$ for some positive constant $C$. In a similar way, $X\gtrsim Y$ denotes $X\geq CY$ for some $C>0$. When the implicit constant has a dependence on some quantities $Z_1, \cdots, Z_n$, we use $X\lec_{Z_1, \cdots, Z_n} Y$ and $X\gtrsim_{Z_1, \cdots, Z_n} Y$. Given an open set $B$, we write $A\Subset B$ when $A$ is compactly embedded in $B$. For a point $z_0=(t_0,x_0) \in \R\times \R^3$ and a positive real number $r>0$, $Q_r$ stands for a parabolic cylinder $(t_0-r^2, 0) \times B_r(x_0)$, where $B_r(x_0)$ is a Euclidean ball centered at $x_0$ with the radius $r$. Similarly, $Q_r^*$ is an extended parabolic cylinder $(t_0-r^2, t_0+r^2) \times B_r(x_0)$. We write the H\"older space on a domain $\mathcal{O}\subset \R\times \R^3$ with respect to the parabolic distance as $C^\al_{\text{par}}(\mathcal{O})$ and the dual space of $W_0^{1,p}(\Om)$ as $W^{-1,p'}(\Om)$ where $p'$ is the H\"older conjugate of $p\in (1,\infty)$.

\subsection*{Acknowledgments}
The author is grateful to Camillo
De Lellis and Elia Bru\`e for helpful discussions and thank Hongjie Dong for suggesting to weaken the smallness assumption. The author has been supported by the National Science Foundation under Grant No. DMS-1926686.

\section{Decomposition of a dissipative solution}\label{sec:dec}
In this section, we decompose a dissipative solution into the principal part and the harmonic part. The principal part is obtained by taking a local version of the Leray projection operator $ - \curl\curl \De^{-1}= \I - \na\De^{-1}\div$ to the solution. We first introduce the local version, which  dates back to \cite{OL03}, to our knowledge.   
\begin{defn}\label{def:llp}(A localized Leray projection operator) Let $\Om$ be a bounded domain in $\R^3$ and $\Om_0$ be an open set with $\Om_0\Subset \Om$. Let $\ph=\ph(x)\in C_c^\infty(\Om;[0,1])$ be a function satisfying $\Om_0 \Subset \{\ph=1\}$. Then, a localized Leray projection operator $\mathbb{P}_{\ph}$ is defined by
\begin{align*}
\mathbb{P}_{\ph} g := - \curl\De^{-1}(\ph \curl g).
\end{align*}
\end{defn}

\begin{rem}\label{rem:llp1} The operator $\mathbb{P}_{\ph}$ maps from $L^p(\Om)$ to $L^p_\si(\Om_0)$ for any $p\in (1,\infty)$, where $L^p_\si(\Om_0) = \{f\in L^p(\Om_0): \div f =0\}$. By the definition, $\mathbb{P}_{\ph}g$ is obviously divergence-free and one can write
\begin{align}
\mathbb{P}_{\ph} g
&= -\curl\De^{-1}\curl( \ph g) + \curl\De^{-1} (\na \ph \times g) 
\nonumber\\
&= g\ph - \na \De^{-1}\div (\ph g) + \curl\De^{-1} (\na \ph \times g). \label{llp1}
\end{align}
It is then easy to see that by the $L^p$-boundedness of Riesz transforms and $\Om_0\Subset \{\ph=1\}$, 
\begin{align}\label{est.v}
\norm{\mathbb{P}_{\ph} g}_{L^p(\Om_0)} 
\lec_{p, \Om_0, \ph} \norm{g\ph}_{L^p(\R^3)} + \norm{\na \ph\times g}_{L^1(\R^3)}
\lec_{\ph} \norm{g}_{L^p(\Om)}.
\end{align}  
Indeed, since $x\in \Om_0$ and $y\in \supp(\na \ph)$ satisfy $|x-y|\geq c$ for some $c=c(\Om_0,\ph)>0$, we have
\begin{align}\label{with.naph}
|\curl\De^{-1}(\na \ph\times g)(x)|
\lec \int \frac{ |\na \ph\times g(y)| }{|x-y|^2}\d y \leq c^{-2}\norm{\na \ph\times g}_{L^1(\R^3)}, \quad\forall x\in \Om_0.
\end{align}
\end{rem}
\begin{rem}\label{rem:llp2} The motivation of the original Leray Projection operator is to eliminate pressure from the Navier-Stokes equations. The localized Leray projection operator still plays the role;
\begin{align*}
\mathbb{P}_\ph \na p =0 
\end{align*}
for any $p\in \mathcal{D}'(\Om)$ in distribution sense. Indeed, for any $\xi\in C_c^\infty(\Om)$, we have
\begin{align*}
(\mathbb{P}_\ph \na p)(\xi)
= \int p \div (\curl(\ph \curl\De^{-1} \xi)) \d x=0.  
\end{align*}
However, it doesn't map a divergence-free vector to itself. Instead, for any divergence-free $g$, the error \begin{align*}
g -\mathbb{P}_{\ph} g
=g\ph -\mathbb{P}_{\ph} g
= \na \De^{-1} (g\cdot \na\ph) - \curl\De^{-1}(\na\ph\times g), \quad \text{ on }\Om_0
\end{align*}
(see \eqref{llp1})
is harmonic and smooth in $\Om_0$; the harmonicity follows from $\na \ph =0$ on $\Om_0$ and for any integer $m\geq 0$,
\begin{align}\label{error.est}
\norm{\na^m(\mathbb{P}_{\ph} g- g)}_{L^\infty(\Om_0)}
&\le \norm{\na^{m+1} \De^{-1}(\na \ph\cdot g )}_{L^p(\Om_0)}+ \norm{\na^{m} \curl\De^{-1} (\na \ph \times g)}_{L^p(\Om_0)}  \nonumber\\
&\lec_{m, \Om_0,\ph} \norm{g}_{L^1(\Om)}.
\end{align}
The last inequality is obtained similar to \eqref{with.naph}. 
\end{rem} 

Now, we introduce the notion of a suitable weak solution to generalized Navier-Stokes equations. 

\begin{defn}[A suitable weak solution]\label{def:suit} We call $(u,p)$ as a suitable weak solution to generalized Navier-Stokes equations
\begin{equation}\label{gen.NS}\begin{split}
\begin{cases}
\pa_t v + \la (v\cdot \na )v+ (v\cdot \na)h  +(h\cdot \na )v + \na q = \De v + f\\
\div v =0
\end{cases}
\end{split}\end{equation}
on a domain $\mathcal{O}\subset  \R\times \R^3$ for $\la\in \R$, divergence-free $h\in L^2_tL^\infty_x(\mathcal{O})$ and divergence-free $f\in L^1_tL^2_x(\mathcal{O}) $ if the following holds
\begin{enumerate}[(i)]
\item $v\in L^\infty_tL^2_x(\mathcal{O})$, $\na v\in L^2(\mathcal{O})$, and $q\in L^\frac32(\mathcal{O})$.   
\item $(v,q)$ solves \eqref{gen.NS} in $\mathcal{O}$ in distribution sense.
\item $(v,q)$ satisfies the local energy inequality 
\begin{align}\label{LEI.v}
(\pa_t -\De)\left(\frac12 |v|^2\right) + |\na v|^2
+ \div \left(\frac 12|v|^2 (\la v+ h)\right) + \div(h\otimes v)\cdot v + \div (qv) - f\cdot v \leq 0
\end{align}
in distribution sense in $\mathcal{O}$. Here and in what follows, $\div (a\otimes b)$ denotes $\pa_j (a_i b_j)$ in the Einstein summation convention. 
\end{enumerate}
\end{defn}

We decompose a dissipative solution based on the following lemma. Indeed, it is the key lemma for Theorem \ref{main.thm}. 
\begin{lem}\label{lem:decomp} Let $\ph\in C_c^\infty(\R^3)$ be a spatial cut-off function satisfying $B_1 \Subset \{\ph =1\}$ and $\supp(\ph)\subset B_2$. 
For any dissipative solution $(u,p)$ to \eqref{NS} on $Q_2$, we have a decomposition of $u= v+h$ in $(-4,0)\times B_1$ such that
\begin{enumerate}[(i)]
\item $v=- \curl\De^{-1}(\ph \curl u)$ is a suitable weak solution to 
\begin{equation}\label{per.NS}\begin{split}
\begin{cases}
\pa_t v + (v\cdot \na )v+ (v\cdot \na)h  +(h\cdot \na )v + \na q = \De v + f\\
\div v =0
\end{cases}
\end{split}\end{equation}
on $(-4,0)\times B_1$ with some pressure $q$ and external force $f$ satisfying the following: $\div f =0$,
\begin{align}
\norm{q}_{L^{\frac r2}(-4,0;L^\frac m2( B_1))}
&\lec_m \norm{u}_{L^r_tL^m_x(Q_2)}^2 \label{est.q}\\
\norm{f}_{L^{\frac s2}(a,b;L^\infty(B_1))} &\lec \norm{u}_{L^s(a,b;L^1(B_2))} + \norm{u}_{L^s(a,b;L^2(B_2))}^2 \label{est.f}
\end{align}
for any $m\in (2,\infty)$, $r, s\in [2,\infty]$, and $(a,b)\subset (-4,0)$, provided that the right hand side of \eqref{est.q} is finite. 

\medskip

\item $h$ is divergence-free and harmonic in $(-4,0)\times B_1$ and satisfies
\begin{align}
\norm{\na^k h}_{L^r(a,b;L^\infty(B_1))}
&\lec_{k} \norm{u}_{L^r(a,b;L^1(B_2))}, \label{est.h}
\end{align}
for any integer $k\geq 0$, $(a,b)\subset (-4,0)$ and $r\in [1,\infty]$. 
\end{enumerate}
\end{lem}

\begin{rem}
In the inequalities \eqref{est.q}-\eqref{est.h}, we ignore the dependence of implicit constants on $\ph$. The proof of the lemma is inspired by \cite{CLRM18, VY20}. 
\end{rem}
\begin{rem}By slight modification of the proof, one can easily see that $f$ has better regularity than \eqref{est.f}; 
\begin{align*}
\norm{\na^k f}_{L^\infty((a,b)\times B_1)}
\lec_k \norm{u}_{L^\infty(a,b;L^1(B_2))}
+\norm{u}_{L^\infty(a,b;L^2(B_2))}^2, \qquad\forall k\in \N\cup\{0\}.
\end{align*}
\end{rem}

\begin{proof} 
By Remark \ref{rem:llp1}, $v$ is divergence-free, which implies that $h=u-v$ is divergence-free in $(-4,0)\times B_1$. Extend $h$ by defining $h = u\ph-v$ on $Q_2$. (Note that $\ph=1$ on $B_1$) Then, by Remark \ref{rem:llp2}, $h$ is harmonic in $(-4,0)\times B_1$ and \eqref{est.h} easily follows from \eqref{error.est}. 
\medskip

\noindent\texttt{Step 1.} $v$ as a weak solution to \eqref{per.NS} in $(-4,0)\times B_1$.

We first note that $v=u-h$ in $(-4,0)\times B_1$ and \eqref{est.h} imply that
\begin{align*}
\norm{v}_{L^\infty_tL^2_x((-4,0)\times B_1)}
+\norm{\na v}_{L^2((-4,0)\times B_1)}
\lec \norm{u}_{L^\infty_tL^2_x(Q_2)}
+\norm{\na u}_{L^2(Q_2)}<+\infty.
\end{align*}

To write \eqref{NS} in terms of $v$, we take the localized Leray projection operator $\mathbb{P}_\ph$ to \eqref{NS}. This can be done rigorously by testing \eqref{NS} with $-\curl(\ph\curl \De^{-1} \xi)$ for any $\xi\in C_c^\infty((-4,0)\times B_1)$. By Remark \ref{rem:llp2}, $\mathbb{P}_\ph \na p =0$. Also, it is easy to see that $\mathbb{P}_\ph \pa_t u= \pa_t v$. Furthermore, a simple computation gives
\begin{equation}\label{err.Deu}\begin{split}
-\curl \De^{-1}(\ph \curl  \De u ) - \De v 
&= 2\curl\De^{-1} (\na \ph \cdot \na \curl u) + \curl\De^{-1} (\De\ph \curl u)\\
&=2\curl \De^{-1} \div (\curl u\otimes \na \ph)
-\curl\De^{-1} (\De\ph \curl u).
\end{split}\end{equation}
These error terms \eqref{err.Deu} will be included in $f$. Indeed, these terms are divergence-free and can be estimated as
\begin{equation}\label{est.fd}\begin{split}
&\norm{2\curl \De^{-1} \div (\curl u\otimes \na \ph)
-\curl\De^{-1} (\De\ph \curl u)}_{L^{s}(a,b;L^\infty(B_1))} \\
&\quad\leq \norm{2\curl \De^{-1}\curl\div (u\otimes \na \ph)-2\curl \De^{-1}\pa_j(\pa_j\na\ph\times u)}_{L^{s}(a,b;L^\infty(B_1))} \\
&\qquad+\norm{\curl\De^{-1}\curl(u \De\ph)- \curl\De^{-1}(\na\De\ph\times u)}_{L^{s}(a,b;L^\infty(B_1))}\\
&\quad\lec \norm{u}_{L^s(a,b;L^1(B_2))}
\end{split}\end{equation}
for any $s\in [1,\infty]$ and $(a,b)\subset(-4,0)$. Here, the second estimate is obtained similar to \eqref{with.naph}. We remark that all equalities in \eqref{err.Deu} are valid on $(-4,0)\times B_1$ in distribution sense, which can be justified by using mollifications and \eqref{est.fd}. 

Using $(\curl a)\times b = (b\cdot \na )a - b_j \na a_j $, on the other hand, we have
\begin{align*}
-\curl \De^{-1}(\ph \curl (u\cdot \na )u )
&= -\curl \De^{-1}(\ph \curl (\om \times u) )\\
&= -\curl \curl\De^{-1} (\ph \om \times u )
+\curl \De^{-1}(\na \ph\times (\om \times u))\\
&= \ph \om \times u - \na \De^{-1} \div  (\ph \om \times u )+\curl \De^{-1}(\na \ph\times (\om \times u)).
\end{align*}
We include the second term in $\na q$ and the last term in $f$.  Since we have
\begin{align*}
\ph \om \times u
=\div\left(u\otimes u \ph -\frac{|u|^2}2\ph\I\right)-u(u\cdot\na) \ph + \frac{|u|^2}2\na\ph.
\end{align*}
by using the $L^m$-boundedness of Riesz transforms and estimating similar to \eqref{with.naph}, we get 
\begin{align*}
\norm{\De^{-1} \div  (\ph \om \times u )}_{L^{\frac r2}(-4,0;L^\frac m2(B_1))} \lec_m \norm{u}_{L^r_tL^m_x(Q_2)}^2,
\end{align*}
for any $r\in [2,\infty]$ and $m\in (2,\infty)$, provided that the right hand side is bounded. On the other hand, the last term $\curl \De^{-1}(\na \ph\times (\om \times u))$ is divergence-free and satisfies
\begin{align*}
\norm{\curl \De^{-1}(\na \ph\times (\om \times u))}_{L^{\frac s2}(a,b;L^\infty(B_1))}
\lec \norm{|u|^2}_{L^{\frac s2}(a,b;L^1(B_2))}
\lec \norm{u}_{L^s(a,b;L^2(B_2))}^2
\end{align*}
for any $s\in [2,\infty]$ and $(a,b)\subset (-4,0)$, so that we put it into $f$. 
Here, we used the following representation 
\begin{align*}
\na\ph \times (\om\times u)&=\na \ph \times \div\left(u\otimes u - \frac{|u|^2}2\I\right)\\
&= \div((\na\ph\times u)\otimes u) + \curl\left(\frac{|u|^2}2 \na\ph\right) - ((u\cdot\na)\na\ph)\times u,
\end{align*}
and estimated it similar to \eqref{with.naph}. Lastly, we consider the first term. Choose another smooth cut-off function $\ph^b=\ph^b(x)$ satisfying $B_1\Subset \{\ph^b=1\}$ and $\supp(\ph^b) \Subset \{\ph=1\}$. On $(-4,0)\times B_1$, we have $\ph = \ph^b =1$, so that 
\begin{align*}
\ph \om \times u
&= \om \times u = \curl v\times v 
+ \curl h \times v + \curl v\times h + \curl h \times h\ph^b\\
&= (v \cdot \na) v  + (v\cdot \na)h + (h\cdot \na)v
-\na\left( \frac{|v|^2}2 + v\cdot h \right)
+ \curl h \times h\ph^b.
\end{align*}
We include the fourth term in $\na q$ because
\begin{align*}
\norm{|v|^2/2 + v\cdot h}_{L^{\frac r2}(-4,0;L^\frac m2( B_1))} 
&\leq \norm{v}_{L^r(-4,0;L^m( B_1))}^2 + \norm{v}_{L^r(-4,0;L^m( B_1))}\norm{h}_{L^r(-4,0;L^m( B_1))}\\
&\lec_m \norm{u}_{L^r_tL^m_x(Q_2)}^2
\end{align*}
for any $r\in [2,\infty]$ and $m\in (2,\infty)$, provided that the right hand side is bounded, by the $L^m$-boundedness of Riesz transforms, \eqref{est.v} and \eqref{est.h}. 
As for the last term, we split it into two parts using $\I = -\curl\De^{-1}\curl + \na \De^{-1} \div$, and put the first part (which is divergence-free) into $f$ and the second part into $\na q$. Indeed, we have  
\begin{align*}
\norm{\De^{-1}\div (\curl h\times h\ph^b)}
_{L^{\frac r2}(-4,0;L^\frac m2( B_1))} 
&\lec \norm{|\na h| |h|}_{L^{\frac r2}_tL^\infty_x((-4,0)\times \supp(\ph^b))}
\lec \norm{u}_{L^r_tL^m_x( Q_2)}^2
\end{align*}
for any $r\in [2,\infty]$ and $m\in (2,\infty)$, provided that the right hand side is bounded. This is because \eqref{est.h} is still valid even on the domain $(-4,0)\times \supp(\ph^b)$ by the choice of $\ph^b$; $\dist(\supp(\ph^b), \supp(\na \ph))>0$. 

Also, for $(t,x)\in (-4,0)\times B_1$ we have
\begin{align*}
|\curl\De^{-1}\curl (\curl h\times h\ph^b)(t,x)|
&\lec
\int \frac {|\curl(\curl h\times h\ph^b)(t,y)|}{|x-y|^2}  \d y\\
&\lec \int_{|x-y|\leq 3} \frac 1{|x-y|^2} \d y 
\norm{\curl(\curl h\times h\ph^b)(t,\cdot)}_{L^\infty(\R^3)}\\
&\lec \norm{h(t,\cdot)}_{W^{2,\infty}(\supp(\ph^b))}^2
\lec \norm{u}_{L^1(B_2)}^2,
\end{align*}
which implies
\begin{align*}
\norm{\curl\De^{-1}\curl (\curl h\times h\ph^b)}_{L^\frac s2(a,b;L^\infty(B_2))}
\lec \norm{u}_{L^s(a,b;L^1(B_2))}^2
\lec \norm{u}_{L^s(a,b;L^2(B_2))}^2
\end{align*}
for any $s\in [2,\infty]$ and $(a,b)\subset (-4,0)$.

To summarize, $v$ solves \eqref{per.NS} in $(-4,0)\times B_1$ in distribution sense with
\begin{align*}
q&= -\De^{-1} \div  (\ph \om \times u ) - \frac 12 |v|^2 - v\cdot h + \De^{-1} \div (\curl h \times h\ph^b)\\
f&= 2\curl \De^{-1} \div (\curl u\otimes \na \ph)
-\curl\De^{-1} (\De\ph \curl u)\\
&\quad-\curl \De^{-1}(\na \ph\times (\om \times u))
+ \curl\De^{-1}\curl(\curl h \times h\ph^b),
\end{align*}
where $f$ is divergence-free and they satisfy the desired estimates \eqref{est.q} and \eqref{est.f}. In particular, $q\in L^\frac32((-4,0)\times B_1)$. 

\medskip

\noindent\texttt{Step 2.} The local energy estimate for $v$.  

Define $ \psi_{\al,\ep}$ as in Definition \ref{def:diss} but choose $\be$ and $\ga$ to be even functions; $\be(t)=\be(-t)$ and $\ga(x)=\ga(-x)$. Set $u_{\al, \ep} := u\ast \psi_{\al,\ep}$ and $v_{\al, \ep} := v\ast \psi_{\al,\ep}$. For any fixed function $\xi\in C_c^\infty((-4,0)\times B_1)\geq 0$, the functions $(u_{\al,\ep} \xi)\ast \psi_{\al,\ep}$ and $(v_{\al,\ep} \xi)\ast \psi_{\al,\ep}$ remain in $C_c^\infty((-4,0)\times B_1)$ for sufficiently small $\al$ and $\ep$. Apply them as a test functions to  \eqref{NS} and \eqref{per.NS}, respectively. Then, we have
\begin{align*}
&\left((\pa_t-\De)  \frac{|u_{\al,\ep}|^2}2+|\na u_{\al,\ep}|^2
+ u_{\al,\ep}\cdot (\div (u\otimes u)\ast\psi_{\al,\ep}) + \div((p\ast \psi_{\al,\ep}) u_{\al,\ep})\right)(\xi)  =0\\
&\left((\pa_t-\De)\frac{|v_{\al,\ep}|^2}2 +|\na v_{\al,\ep}|^2
+ v_{\al,\ep}\cdot (\div (v\otimes v+ v\otimes h + h\otimes v)\ast\psi_{\al,\ep})\right)(\xi) \\
&\hspace{6.1cm}+ \left(\div((q\ast \psi_{\al,\ep}) v_{\al,\ep}) - v_{\al,\ep}\cdot (f\ast\psi_{\al,\ep})\right)(\xi) =0.
\end{align*}
Sending $\al\to 0$ and then $\ep\to 0$, it follows that
\begin{align*}
&\left((\pa_t-\De)  \frac{|u|^2}2
+|\na u|^2+ \div \left(\frac{|u|^2}2u\right) 
+\langle \div (pu) \rangle 
\right)(\xi)  =-\lim_{\ep\to 0} \mu_\ep(\xi)
\\
&\left((\pa_t-\De)\frac{|v|^2}2
+|\na v|^2+ \div \left(\frac{|v|^2}2(v+h)\right) + \div (h\otimes v) \cdot v + \div(qv)- f\cdot v\right)(\xi) =-\lim_{\ep\to 0} \eta_\ep(\xi) ,
\end{align*}
where $\mu_\ep = u_\ep \cdot (\div (u\otimes u) \ast \ga_\ep)-\div (|u|^2 u/2)$ and $\eta_\ep = v_\ep \cdot (\div (v\otimes v) \ast \ga_\ep)-\div (|v|^2 v/2)$.
Since $\underset{\ep\to 0}{\lim} \, (\mu_\ep-\eta_\ep)(\xi) =0$ by \cite[Lemma 3.7]{CLRM18} and the local energy inequality for $u$ gives $\lim_{\ep\to 0} \mu_\ep(\xi) \geq 0$, we have $\underset{\ep\to 0}{\lim}\, \eta_\ep(\xi) \geq 0$, which implies the local energy inequality for $v$.
\end{proof}

\section{$\e$-regularity criterion for a perturbed Navier-Stokes equations}\label{sec:reg}
In this section, we obtain an $\e$-regularity criterion for a 
suitable weak solution to perturbed Navier-Stokes equations with some external force, obtained in the previous section:

\begin{thm}\label{thm:ep.reg} For any $r,m\in (2,\infty]$ with $2/r+3/m <2$, there exist $\e_0=\e_0(r,m)>0$ and a universal constant $C>0$ such that if a suitable weak solution $(v,q)$ to \eqref{per.NS} on $Q_1$ for some $h$ and $f$ with $\div h=\div f =0$ satisfies
\begin{align} 
\label{small.e0}
\norm{v}_{L^r_tL^m_x(Q_1)}
+ \norm{q}_{L^{\frac r2}_tL^{\frac m2}_x(Q_1)} + \norm{h}_{L^r_tW^{1,\infty}_x(Q_1)}
+ \norm{f}_{L^\frac r2_tL^\infty_x(Q_1)} \le \e_0,
\end{align} 
then we have
\begin{align*}
\norm{v}_{C^{\al}_{\text{par}}(Q_{1/2})} \leq C 
\end{align*}
for some $\al\in (0,1/2-1/r)$.
\end{thm}

\begin{rem} In the case of $f=0$ and $r=m=3$ and under the smallness assumption on $h$ in $L^5(Q_1)$ or the assumption $h\in L^m(Q_1)$ for some $m>5$, $\e$-regularity criterions are obtained in \cite{JS14, KMT20, BP20}. The proof of Theorem \ref{thm:ep.reg} follows the scheme of the proofs in \cite{Lin98, JS14}.
\end{rem}

The proof of Theorem \ref{thm:ep.reg} is based on the following oscillation lemma.
\begin{lem}[Oscillation lemma]\label{lem:osc} Suppose $r,m \in (2,\infty)$ satisfy
\begin{align}\label{con.rm}
\frac 2r + \frac 3m = 2-\si \qquad \text{ for some }0<\si< \min\left(\frac16, \frac1r, \frac1m\right)  .
\end{align}
For any $\th\in (0, 1/3)$ and $\la\in (0,1]$, we can find positive constants $\e_1=\e_1(\th,r,m)$, $C=C(r,m)$, and $\be< \min(1/2, 1-2/r)$ such that  if a suitable weak solution $(v,q)$ to 
\begin{equation}\label{per.NS1}\begin{split}
\begin{cases}
\pa_t v + \la (v\cdot \na )v+ (v\cdot \na)h  +(h\cdot \na )v + \na q = \De v + f\\
\div v =0
\end{cases}
\end{split}\end{equation}
on $Q_1$ for some divergence-free functions $h$ and $f$ satisfies $|(v)_{Q_1}| \leq 1$ and
\begin{align}\label{small.e1}
\frac{\norm{v-(v)_{Q_1}}_{L^r_tL^m_x(Q_1)}}{|B_1|^{\frac 1m}} + \frac{\norm{q-(q)_{B_1}}_{L^\frac r2_tL^\frac m2_x(Q_1)}}{|B_1|^{\frac 2m}} 
+\norm{h}_{L^r_tW_x^{1,\infty}(Q_1)}
+ \norm{f}_{L^\frac r2_tL^\infty_x(Q_1)} \leq \e_1 
\end{align}
then 
\begin{align*}
&\frac{\norm{v-(v)_{Q_\th}}_{L^r_tL^m_x(Q_\th)}}{\th^{2-\si}|B_1|^{\frac 1m}} + \frac{\th\norm{q-(q)_{B_\th}}_{L^\frac r2_tL^\frac m2_x(Q_\th)}}{\th^{2(2-\si)}|B_1|^{\frac 2m}} 
\leq C \th^\be
W_{r,m}(v,q,h,f,1)
\end{align*}
where $W_{r,m}(v,q,h,f,1)$ denotes the left hand side of the inequality \eqref{small.e1}.
\end{lem}

We first introduce preliminary lemmas. 
\begin{lem}\label{lem:itr}
Suppose that $A$ is a non-negative non-decreasing bounded function on $[0,1]$ and satisfies
\begin{align*}
A(\rho) \leq \th A(R) + \frac{M}{(R-\rho)^\al} \quad\text{ for any}\quad\frac 34\leq \rho\leq R< 1
\end{align*}
for some $\th \in (0,1)$, $M\in (0,\infty)$, $\al>0$. Then,  $A$ satisfies 
\begin{align*}
\sup_{\rho\in [0,3/4]}A(\rho) \leq C(\th, M, \al)
\end{align*}
for some positive constant $C(\th, M, \al)$ depending only on $\th, M, \al$. 
\end{lem}
\noindent For the proof, see \cite{Evans86}. Also, the following lemma will be used to relax the smallness assumption from the classical ones. Similar estimates appear in \cite{DoWa21}. 

\begin{lem}\label{pre.relax} Let $3/4\le \overline \rho<R<1$, and $
P_o= (-\De)^{-1}\div\div \left( ( G\otimes U + U\otimes G)\chi_{B_R}\right),
$
where $\chi_{E}$ is the characteristic function of a set $E$. 
For any $r,m$ satisfying the same assumption as in Lemma \ref{lem:osc}, we have
\begin{align}
\norm{U}_{L^{2r'}_t L^{2m'}_x(Q_R)} 
&\lec \norm{U}_{L^r_tL^m_x(Q_R)}^\upgamma
\norm{U}_{\mathcal{E}(Q_R)}^{1-\upgamma}  \label{est.Vi}\\
\norm{P_o}_{L^{{\overline r}'}(-{\overline\rho}^2,0;L^{{\overline m}'} (\R^3))}
&\lec_{r,m} \norm{G}_{L^r_tL^m_x(Q_R)}\norm{U}_{L^r_tL^m_x(Q_R)}^{2\upgamma} \norm{U}_{\mathcal{E}(Q_R)}^{1-2\upgamma}\label{est.Pio}
\end{align}
where $\norm{U}_{\mathcal{E}(Q_R)}^2:=\norm{U}_{L^\infty_tL^2_x(Q_R)}^2 + \norm{\na U}_{L^2(Q_R)}^2$,  the parameters $\upgamma, \overline r, \overline m$ are defined by
\begin{align*}
\upgamma = \frac{\si}{1-2\si}, \quad
\overline r =\frac{2r (1-3\si)}{r (1-2\si)-1}, \quad
\overline m = \frac{2m (1-3\si)}{m (1-2\si)-1},
\end{align*}
and $r'$, $m'$, $\overline r'$, $\overline m'$  are the H\"older conjugates of $r$, $m$, $\overline r$, $\overline m$,  respectively.
\end{lem}
\begin{proof} The proof is based on the standard interpolations. Since \eqref{con.rm} implies that
\begin{align*}
\frac 1r +\frac 2m >1, \quad \frac 1r + \frac 1m \geq \frac 23
\end{align*}
we have $\upgamma\in (0,1)$, $\overline r\in[2,\infty]$, and $\overline m \in [2,6]$. Also, we note that $(\overline r, \overline m)$ satisfies 
\begin{align*}
\frac 2{\overline r} + \frac 3{\overline m} = \frac 32, \quad
\frac 1{2r'} = \frac \upgamma r + \frac{(1-\upgamma)}{\overline r}, \quad
\frac 1{2m'} = \frac \upgamma m + \frac{(1-\upgamma)}{\overline m}.
\end{align*}
For such $(\overline r, \overline m)$, we recall that
\begin{align}\label{int.energy}
\norm{U}_{L^{\overline r}_tL^{\overline m}_x(Q_R)}
\lec \norm{U}_{\mathcal{E}(Q_R)},
\end{align}
which can be obtained by the interpolations. Then, \eqref{est.Vi} follows from the interpolations;
\begin{align*}
\norm{U}_{L^{2r'}_t L^{2m'}_x(Q_R)} 
&\lec \norm{U}_{L^r_tL^m_x(Q_R)}^\upgamma 
\norm{U}_{L^{\overline r}_tL^{\overline m}_x(Q_R)}^{1-\upgamma}
\lec \norm{U}_{L^r_tL^m_x(Q_R)}^\upgamma
\norm{U}_{\mathcal{E}(Q_R)}^{1-\upgamma}.
\end{align*}
To estimate \eqref{est.Pio}, we use the $L^p$-boundedness of Riesz-transforms and interpolations;
\begin{align*}
\norm{P_o}_{L^{{\overline r}'}(-{\overline\rho}^2,0;L^{{\overline m}'} (\R^3))}
&\lec_{\overline m} \norm{U\otimes G + G\otimes U}_{L^{{\overline r}'}_tL^{{\overline m}'}_x(Q_R)}
\lec_{\overline m} \norm{G}_{L^r_tL^m_x(Q_R)} \norm{U}_{L^s_t L^n_x(Q_R)}\\
&\lec_{\overline m} \norm{G}_{L^r_tL^m_x(Q_R)}\norm{U}_{L^r_tL^m_x(Q_R)}^{2\upgamma} \norm{U}_{\mathcal{E}(Q_R)}^{1-2\upgamma}
\end{align*}
where $1/s = 1- 1/r - 1/\overline r$ and $1/n = 1- 1/m - 1/\overline m$, and the last inequality is obtained similar to \eqref{est.Vi}.
\end{proof}

\medskip

We now give a proof of the oscillation lemma. 
\begin{proof}[Proof of Lemma \ref{lem:osc}.] 
The proof is based on the contradiction argument. Suppose that there exist a sequence $(v_i,q_i)$ of suitable weak solutions to \eqref{per.NS1} with some $\la\in (0,1]$ and $(h_i,f_i)$ such that $\div h_i=\div f_i =0$, $|(v_i)_{Q_1}|\le 1$, and
\begin{align*}
W_{r,m}(v_i, q_i, h_i, f_i,1)  = \e_i \searrow 0
\end{align*}
as $i$ goes to infinity, but there exists $\th\in (0,1/3)$ such that
\begin{align}\label{osc.est}
&\frac{\norm{v_i-(v_i)_{Q_\th}}_{L^r_tL^m_x(Q_\th)}}{\th^{2-\si}|B_1|^{\frac 1m}} + \frac{\th\norm{q_i-(q_i)_{B_\th}}_{L^\frac r2_tL^\frac m2_x(Q_\th)}}{\th^{2(2-\si)}|B_1|^{\frac 2m}} 
\geq C \th^\be\e_i, \quad i\in \mathbb{N}\cup\{0\}
\end{align}
where $\be\in (0,1)$ and $C>0$ will be chosen later.

Define $(V_i, P_i, H_i, F_i)$ by
\begin{align*}
V_i = \frac{v_i -(v_i)_{Q_1}}{\e_i}, \quad
P_i = \frac{q_i - (q_i)_{B_1}}{\e_i},\quad H_i = \frac{h_i}{\e_i}, \quad
F_i = \frac{f_i}{\e_i}.
\end{align*} 
Then, it satisfies $\div V_i =\div H_i = \div F_i=0$, 
\begin{align*}
\pa_t V_i + \div \left((\la V_i+H_i)\otimes (\e_i V_i+(v_i)_{Q_1}) + \e_i V_i\otimes H_i\right) + \na P_i = \De V_i + F_i,
\end{align*}
and 
\begin{equation}\label{LEI.V}\begin{split}
(\pa_t -\De) \left(\frac12 |V_i|^2\right)
&+|\na V_i|^2
+\div \left(\frac12|V_i|^2 (\la \e_i V_i + \la (v_i)_{Q_1} +\e_i H_i)\right)\\
& + (((\e_iV_i + (v_i)_{Q_1})\cdot\na)H_i) \cdot V_i + 
\div (P_i V_i) -  F_i V_i\le 0
\end{split}\end{equation}
on $Q_1$ in distribution senses. In particular, $P_i$ solves
\begin{align}\label{eqn.Pi}
-\De P_i =  \e_i\div\div \left( \la V_i\otimes V_i + H_i\otimes V_i + V_i\otimes H_i\right) 
\end{align}
on $Q_1$. Here, we used $\div\div(V_i \otimes \la(v_i)_{Q_1})=\div\div(H_i \otimes (v_i)_{Q_1})=0$ because $(v_i)_{Q_1}$ is a constant and both $V_i$ and $H_i$ are divergence-free. Furthermore, by the assumption, we have
\begin{align}\label{bdd.i}
\norm{V_i}_{L^r_tL^m_x(Q_1)} + \norm{P_i}_{L^{\frac r2}_tL^{\frac m2}_x(Q_1)}
+\norm{H_i}_{L^r_tW^{1,\infty}_x(Q_1)} + \norm{F_i}_{L^{\frac r2}_tL^\infty_x(Q_1)} \lec 1.
\end{align}

\medskip

\noindent\texttt{Step 1.} Passage to the limit.

Using \eqref{LEI.V} and Lemma \ref{lem:itr}-\ref{pre.relax}, we first obtain a uniform bound of $V_i$ in the energy space norm. Let $3/4\le \rho<R<1$ and $\overline \rho= (\rho+R)/2$. We decompose $P_i$ on $Q_R$ into $P_i= P_{io} + P_{ih}$ where
\begin{align*}
P_{io} = \e_i(-\De)^{-1}\div\div \left( (\la V_i\otimes V_i +  H_i\otimes V_i + V_i\otimes H_i)\chi_{B_R}\right),
\end{align*}
for the characteristic function $\chi_{E}$ of a set $E$. Then, $P_{ih}$ is harmonic on $B_R$ so that we have
\begin{align}\label{est.Pih}
\norm{P_{ih}(\cdot, t)}_{L^\infty(B_{\overline \rho})}
\lec \frac {\norm{P_{ih}(\cdot,t)}_{L^1(B_R)}}{(R-\overline \rho)^3}
\leq  \frac {\norm{P_i(\cdot,t)}_{L^1(B_R)}}{(R-\rho)^3}
+\frac {\norm{P_{io}(\cdot,t)}_{L^1(B_R)}}{(R-\rho)^3}
\end{align}
for a.e. $t\in (-{\overline \rho}^2,0)$. Then, by Lemma \ref{pre.relax}, \eqref{int.energy}, \eqref{bdd.i} and \eqref{est.Pih}, we have
\begin{align*}
\int_{Q_R} |V_i|^2 ( \la|V_i|+ |H_i| + |\na H_i|)\, \d x \d t 
&\lec (\norm{V_i}_{L^r_tL^m_x(Q_R)} + \norm{H_i}_{L^r_tW^{1,\infty}_x(Q_R)})
\norm{V_i}_{L^{2r'}_tL^{2m'}_x(Q_R)}^2\\
&\lec\norm{V_i}_{L^r_tL^m_x(Q_R)}^{2\upgamma}\norm{V_i}_{\mathcal{E}(Q_R)}^{2(1-\upgamma)}
\lec \norm{V_i}_{\mathcal{E}(Q_R)}^{2(1-\upgamma)}.
\end{align*}
and 
\begin{align*}
&\int_{Q_{\overline \rho}} |P_i||V_i| \,\d x \d t
\lec \norm{P_{io}}_{L^{\overline r'}_tL^{\overline m'}_x(Q_{\overline \rho})}\norm{V_i}_{L^{\overline r}_tL^{\overline{m}}_x(Q_R)} 
 +\norm{P_{ih}}_{L^1_tL^\infty_x(Q_{\overline \rho})}\norm{V_i}_{L^\infty_tL^1_x(Q_R)}\\
&\quad\lec_{r,m} \e_i(\norm{V_i}_{L^r_tL^m_x(Q_R)}+\norm{H_i}_{L^r_tL^\infty_x(Q_R)})\norm{V_i}_{L^r_tL^m_x(Q_R)}^{2\upgamma} \norm{V_i}_{\mathcal{E}(Q_R)}^{2-2\upgamma}\\
&\hspace{8cm} + \frac {\norm{V_i}_{\mathcal{E}(Q_R)}}{(R-\rho)^3}( \norm{P_i}_{L^1(Q_R)} +\norm{P_{io}}_{L^1(-{\overline \rho}^2,0;L^1(\R^3))})\\
&\quad\lec_{r,m}\frac {\e_i \norm{V_i}_{\mathcal{E}(Q_R)}^{2-2\upgamma}}{(R-\rho)^3}(\norm{V_i}_{L^r_tL^m_x(Q_R)}+\norm{H_i}_{L^r_tL^\infty_x(Q_R)}) \norm{V_i}_{L^r_tL^m_x(Q_R)}^{2\upgamma}+ \frac {\norm{V_i}_{\mathcal{E}(Q_R)}}{(R-\rho)^3}\norm{P_i}_{L^{\frac r2}_tL^{\frac m2}_x(Q_1)}\\
&\quad\lec_{r,m} \frac {\e_i}{(R-\rho)^3}\norm{V_i}_{\mathcal{E}(Q_R)}^{2-2\upgamma} +  \frac {\norm{V_i}_{\mathcal{E}(Q_R)}}{(R-\rho)^3}.
\end{align*}

Since $V_i(t,\cdot)$ has weak continuity in $L^2(Q_1)$ (by redefining a measure zero set in time if necessary), testing a smooth cut-off $\xi$ with $\xi=1$ on $Q_\rho$ and $\supp(\xi)\subset Q_{\overline \rho}^*$ on \eqref{LEI.V}, we have 
\begin{align*}
\norm{V_i}_{\mathcal{E}(Q_\rho)}^2
&\lec \frac{\norm{V_i}_{L^2(Q_1)}^2 }{(R-\rho)^2}
+ \frac{\e_i}{R-\rho} \int_{Q_R} |V_i|^2( \la|V_i| +|H_i| + |\na H_i|)\, \d x \d t  + \frac 1{R-\rho} \int_{Q_{\overline \rho}} |P_i||V_i| \,\d x \d t\\
 &\quad+\norm{H_i}_{L^r_tW^{1,\infty}_x(Q_1)}\norm{V_i}_{L^{r'}_tL^1_x(Q_R)} + \norm{F_i}_{L^1_tL^\infty_x(Q_1)}\norm{V_i}_{\mathcal{E}(Q_R)}\\
 &\leq \frac 12 \norm{V_i}_{\mathcal{E}(Q_R)}^{2} + \frac{\tilde M(r,m) }{(R-\rho)^{4\max(2,1/\upgamma) }}
\end{align*}
for sufficiently large $i$ to have $\e_i\leq 1$ for some positive constant $\tilde M(r,m)$ depending only on $r$ and $m$. The last inequality follows from \eqref{bdd.i} and Young's inequality. Then, it follows from Lemma \ref{lem:itr} that
\begin{align*}
\norm{V_i}_{\mathcal{E}(Q_{3/4})} \lec_{r,m} 1.
\end{align*}
for sufficiently large $i$.
Using interpolation inequality, we then have
\begin{align}\label{bdd.Vi}
\norm{V_i}_{L^{\frac{10}3}(Q_{3/4})} + \norm{V_i}_{L^{\td r}_tL^{\td m}_x(Q_{3/4})} \lec_{r,m} 1,
\end{align}
where $\td r = {2(2-\si)r}/3>r$ and $\td m= {2(2-\si)m}/3>m$.
Therefore, by the compactness argument, \eqref{bdd.i} and \eqref{bdd.Vi} imply the following convergences
\begin{align*}
V_i \to& V \quad\text{strongly in }L^3(Q_{3/4})\text{ and } L^r_tL^m_x(Q_{3/4}), \quad (v_i)_{Q_1} \to \ka \quad\text{in } \R^3\\
P_i &\rightharpoonup P \quad\text{weakly in }L^{\frac r2}_t L^{\frac m2}_x(Q_1), \quad
F_i \overset{\ast}{\rightharpoonup} F \quad\text{weakly in } L^\frac r2_tL^\infty_x(Q_1)\\
&\quad H_i \overset{\ast}{\rightharpoonup} H \quad\text{and} \quad 
\na H_i \overset{\ast}{\rightharpoonup} H^{(1)} \quad\text{weakly in }L^r_tL^\infty_x(Q_1)
\end{align*}
up to subsequence, where the limit pair $(V,P,\ka, H, H^{(1)}, F)$ satisfies
\begin{align*}
&\norm{V}_{L^r_tL^m_x(Q_{3/4})}
 +\, \norm{P}_{L^\frac r2_t L^\frac m2 _x(Q_1)}
+\norm{H}_{L^r_tL^{\infty}_x(Q_1)}+\norm{H^{(1)}}_{L^r_tL^{\infty}_x(Q_1)} + \norm{F}_{L^\frac r2_tL^\infty_x(Q_1)} \lec 1,\\
&\hspace{3cm}\norm{V}_{L^3(Q_{3/4})} \lec_{r,m} 1, \quad
|\ka |\le 1, \quad \div H=\div F =0.
\end{align*}
Here, we denote $(\na g)_{ij} = \pa_i g_j$. 
Also, the limit pair solves a generalized Stokes system on $Q_{3/4}$
\begin{equation}\label{Stokes.V}
\begin{cases}
\pa_t V +\na P -\De V 
= -\div ( V\otimes \la\ka) - \ka H^{(1)} + F\\
\div V =0
\end{cases}
\end{equation}
in distribution sense, and $P$ satisfies $-\De P=0$ on $Q_{3/4}$. Here, $(\ka H^{(1)})_m = \ka_\ell H^{(1)}_{\ell m}$. 

\bigskip

\noindent\texttt{Step 2.} Regularity of the generalized Stokes system. 

Let $1/2\leq \rho<R< 3/4$. Let $\zeta$ be a smooth cut-off such that $\zeta =1$ on $Q_{\overline \rho}$, ${\overline \rho} = (r+R)/2$, and $\supp(\zeta)\subset Q_{R}^*$. 
We then rewrite \eqref{Stokes.V} on $Q_{\overline \rho}$ as 
\begin{align*}
(\pa_t -\De)(V\zeta)
=&\, \underbrace{-\div( V\otimes \la\ka \zeta ) 
}_{= I_1}
+\underbrace{F\zeta - \ka H^{(1)} \zeta  -\zeta\na P}_{=I_2}
\end{align*}
and decompose $V$ into $V= V\zeta=V_1 + V_2 + V_3 $ on $Q_{\overline \rho}$ by defining
\begin{align*}
V_i(t,x) = \int_{-1}^t e^{(t-\tau)\De} I_i(\tau,x) \d \tau \quad \text{on }Q_{\overline \rho}, \quad \forall i =1,2.
\end{align*}
By Lemma \ref{lem:semigp}, we have	
\begin{align*}
\norm{V_1}_{L^q(Q_{\overline \rho})}
\lec_{p,q} {\norm{|\la\ka| |V|\zeta}_{L^p([-1,0]\times \R^3)}}
\lec \norm{V}_{L^p(Q_R)}
\end{align*}
for $1\le p<q<\infty$ such that $\frac 1q = \frac 1p-\frac15 $.
To estimate $V_2$, we note that the harmonicity of $P$ on $Q_{3/4}$ implies that 
\begin{align*}
\norm{\na P}_{L^\frac r2_tL^\infty_x(Q_R)}
\lec_R \norm{P}_{L^\frac r2_tL^\frac m2_x (Q_{3/4})}
\lec_R 1
\end{align*}
by elliptic estimates.
Then, by Lemma \ref{lem:semigp2}, we get
\begin{align*}
\norm{V_2}_{C^\be_{\text{par}}(Q_{\overline \rho})}
\lec \norm{F}_{L^\frac r2_tL^\infty_x(Q_R)}
+\norm{ H^{(1)}}_{L^r_tL^{\infty}_x(Q_R)} 
+ \norm{\na P}_{L^\frac r2_tL^\infty_x(Q_R)}
\lec 1
\end{align*}
for some $\be\in (0,\min(1-2/r, 1/2))$. Lastly, since $V_3$ solves $(\pa_t-\De)V_3 =0$ on $Q_{\overline \rho}$, it is smooth in $Q_r$; in particular, $\norm{V_3}_{C^1_{t,x}(Q_r)} \lec_{r,R} \norm{V_3}_{L^1(Q_{\overline \rho})}\lec 1$. 

Since $q>p$, starting with $p=3$, after a finite number of iterations, we get $\norm{V}_{L^{\frac{15}2}(Q_{5/8})} \lec_{r,m} 1$ and hence by Lemma \ref{lem:semigp2}
\begin{align*}
\norm{V}_{C^\be_\text{par}(Q_{1/2})} \lec_{r,m} 1.
\end{align*}
In particular, we have
\begin{align*}
|V(t,x) -V(s,y)| \lec_{r,m} (|x-y| + |t-s|^{\frac12})^\be
\end{align*}
for any $(t,x),(s,y)\in Q_{1/2}$ for some $\be\in (0,\min (1/2, 1-2/r))$. This and the strong convergence of $V_i$ in $L^r_tL^m_x(Q_{3/4})$ then implies
\begin{align}\label{osc.vi}
&\frac 1{\e_i}\left(\dint_{-\th^2}^0\left(\dint_{B_\th} |v_i-(v_i)_{Q_\th}|^m\d x\right)^{\frac r m} \d t \right)^\frac1r\\
&= \left(\dint_{-\th^2}^0\left(\dint_{B_\th} \left|\dint_{Q_\th} V_i(t,x) -V_i(s,y) \,\d y \d s \right |^m\d x\right)^{\frac r m} \d t \right)^\frac1r
\lec_{r,m} \th^\be 
\end{align}
for sufficiently large $i$. 

Recall that $P_i$ solves \eqref{eqn.Pi} and can be decomposed into $P_{io}$ and $P_{ih}$ on $Q_{3/4}$ where
\begin{align*}
P_{io}:= \e_i(-\De)^{-1}\div\div((\la  V_i\otimes V_i + H_i\otimes V_i +  V_i\otimes H_i)\chi_{Q_{3/4}})
\end{align*}
and hence $\De P_{ih} = 0$ on $Q_{3/4}$. Using \eqref{bdd.i}, one can easily see that $P_{io}$ strongly converges to $0$ in $L^\frac r2_tL^\frac m2_x(Q_{3/4})$ as $i$ goes to infinity. Therefore, we get
\begin{align*}
&\frac{\th}{\e_i}\left(\dint_{-\th^2}^0\left(\dint_{B_\th} |q_i-(q_i)_{B_\th}|^\frac m2 \d x\right)^\frac rm \d t 
\right)^\frac 2r\\
&\quad\lec \th^{-3 + 2\si}
\Norm{\dint_{B_\th} P_i- P_i(\cdot,y) \d y }_{L^\frac r2_t L^\frac m2_x(Q_\th)}\lec \th^{-3 + 2\si}\Norm{\dint_{B_\th} |P_{ih}- P_{ih}(\cdot,y)| \d y }_{L^\frac r2_t L^\frac m2_x(Q_\th)}+ \th^\be\\
&\quad\lec  \th^{-3 + 2\si}\norm{\norm{\na P_{ih}(t,\cdot)}_{L^\infty(B_\frac12)}\th}_{L^\frac r2_t L^\frac m2_x(Q_\th)}+ \th^\be
\lec \th^{2(1-\si)+\frac 6m}\norm{P_i}_{L^\frac r2_tL^\frac m 2_x(Q_{3/4})}+ \th^\be
\lec \th^\be.
\end{align*}
for sufficiently large $i$, by elliptic estimates. The last inequality follows from $2(1-\si) + \frac 6m >2(1-\si)>\be$. 
Combining it with \eqref{osc.vi}, we have
\begin{align*}
\frac{\norm{v_i-(v_i)_{Q_\th}}_{L^r_tL^m_x(Q_\th)}}{\th^{2-\si}|B_1|^{\frac 1m}} + \frac{\th\norm{q_i-(q_i)_{B_\th}}_{L^\frac r2_tL^\frac m2_x(Q_\th)}}{\th^{2(2-\si)}|B_1|^{\frac 2m}} 
\leq \td C \th^{\be} \e_i
\end{align*}
for sufficiently large $i$ for some $\be<\min(1/2, 1-2/r)$ and $\td C= \td C(r,m)>0$. Finally, if we choose $C> \td C$, it contradicts to \eqref{osc.est}.

\end{proof}

Lastly, we prove Theorem \ref{thm:ep.reg} by iterating Lemma \ref{lem:osc}.

\begin{proof}[Proof of Theorem \ref{thm:ep.reg}.] \quad
WLOG, we can assume $(r,m)$ satisfies the assumptions in Lemma \ref{lem:osc}. 
Otherwise, in the case of $r<\infty$, choose $\si\in (0, \min(1/6, 1/r, 1/m))$ and then $\widetilde m\in (2, m)$, so that $(r,\widetilde m)$ satisfies \eqref{con.rm}. The smallness assumption \eqref{small.e0} with the replacement of $m$ by $\widetilde m$ then follows by H\"older's inequality, adjusting $\e_0(r, m)$ if necessary. In the case of $r=\infty$, one can choose $\widetilde r \in (2,\infty)$ so that $2/{\widetilde r} + 3/m <2$, and \eqref{small.e0} with the replacement of $r$ by $\widetilde r$ is again valid, adjusting $\e_0(r,m)$. Since $\widetilde r \in (2,\infty)$, we can find $\widetilde m$ as in the case of $r<\infty$ and work with $(\widetilde r, \widetilde m)$. Note that $\al<1/2-1/{\widetilde r}< 1/2$.
\\

\noindent\texttt{Step 1.} Iteration of Lemma \ref{lem:osc}. 

For the convenience, for $z_0 = (t_0,x_0)$, we denote
\begin{align*}
U_{r,m}(v,q,\rho,z_0)
&:=\frac{\norm{v-(v)_{Q_\rho(z_0)}}_{L^r_tL^m_x(Q_\rho(z_0))}}{\rho^{2-\si}|B_1|^{\frac 1m}} 
+ \frac{\rho \norm{q-(q)_{B_\rho(x_0)}}_{L^\frac r2_tL^\frac m2_x(Q_\rho(z_0))}}{\rho^{2(2-\si)}|B_1|^\frac2m}\\
W_{r,m}(v,q,h,f,r, z_0)
&:=U_{r,m}(v,q,\rho, z_0) + \rho^{1-\frac2r}\norm{h}_{L^r_tW^{1,\infty}_x(Q_\rho(z_0))} + \rho^{2-\frac4r}\norm{f}_{L^\frac r2_tL^\infty_x(Q_\rho(z_0))}
\end{align*}
and we suppress $z_0$ when $z_0=0$. We first prove the following claim. 
\medskip

\noindent{\bf Claim: }There exist $\th=\th(r,m)\in (0,1/3)$ and $\e_2=\e_2(r,m) \in(0, \min(\e_1,1)]$, where $\e_1=\e_1(\th,r,m)$ is determined by Lemma \ref{lem:osc}, such that if a suitable weak solution $(v,q)$ to \eqref{per.NS} with some divergence-free $h$ and $f$ satisfies $|(v)_{Q_1}| \leq 1/2$ and $W_{r,m}(v,q,h,f,1)\leq \e_2$, then 
\begin{align} \label{ind.hyp}
 |(v)_{Q_{\th^k}}|\leq 1,  \quad
W_{r,m}(v,q,h,f,\th^k)   \leq \th^{k\al }\e_2, \qquad\forall k\in \mathbb{N}\cup \{0\}
\end{align}
for some $\al\in (0,1/2-1/r)$. In particular, 
\begin{align*}
\left(\dint_{Q_{\th^k}} |v-(v)_{Q_{\th^k}}|^2 \d x\d t \right)^\frac12 \leq \th^{k\al} \e_2, \qquad\forall k\in \mathbb{N}\cup \{0\}.
\end{align*} 

\medskip
\begin{proof}[Proof of Claim.]
We prove the claim by the induction argument.
One can easily check that \eqref{ind.hyp} holds true when $k=0$. 

Now, we suppose that \eqref{ind.hyp} holds for any integer $k\in [0,n]$ for some $n\in \mathbb{N}\cup \{0\}$. Define the rescaled pair
\begin{align*}
&v_R(t,x) := v(R^2 t , Rx),  \quad
q_R(t,x) := R q(R^2 t, Rx),\\
&h_R(t,x) := R h(R^2 t, Rx),\quad
f_R(t,x) := R^2 f(R^2 t, Rx)
\end{align*}
for $R>0$. Then, for each $R\in (0,1]$, the rescaled pair $(v_R, q_R, h_R, f_R)$ is a suitable weak solution to \eqref{per.NS1} on $Q_1$ for $\la = R$ and $\div h_R = \div f_R =0$. Set $R=\th^n$ for $\th\in (0,1/3)$. A simple computation gives
\begin{align*}
|(v_{\th^n})_{Q_1}|&= |(v)_{Q_{\th^n}}| \leq 1\\
W_{r,m}(v_{\th^n}, q_{\th^n}, h_{\th^n}, f_{\th^n}, 1)
&\le W_{r,m}(v,q, h, f,\th^n)
\le \th^{n\al} \e_2\le \e_1.
\end{align*}
Therefore, applying Lemma \ref{lem:osc}, we have positive constants $\be\in (0,\min(1/2, 1-2/r))$ and $C=C(r,m)$ such that
\begin{align*}
U_{r,m}(v,q,\th^{n+1})
&=U_{r,m}(v_{\th^n}, q_{\th^n}, \th)
\leq C\th^\be W_{r,m}(v_{\th^n},q_{\th^n},h_{\th^n},f_{\th^n}, 1)\\
&\leq C\th^\be  W_{r,m}(v,q,h,f,\th^n) \leq \frac12\th^{(n+1)\al}\e_2 
\end{align*}
if $\th$ is small enough to satisfy $2C\th^\be \leq \th^{\frac{\be}2}$ and $\al\leq \be/2<\min(1/2-1/r, 1/4)$. 
Since we also have
\begin{align*}
&\th^{(n+1)\left(1-\frac2r\right)}\norm{h}_{L^r_tW^{1,\infty}_x(Q_{\th^{n+1}})} 
+ \th^{2(n+1)\left(1-\frac2r\right)} \norm{f}_{L^\frac r2_tL^\infty_x(Q_{\th^{n+1}})}\\
&\hspace{2cm}\leq \th^{(n+1)\left(1-\frac2r\right)} \left(\norm{h}_{L^r_tW^{1,\infty}_x(Q_1)} + \norm{f}_{ L^\frac r2_tL^\infty_x(Q_1)}\right)
\leq  \th^{(n+1)\al} \th^{\frac12-\frac1r} \e_2
\leq \frac12\th^{(n+1)\al}\e_2
\end{align*}
if $ \th^{\frac12-\frac1r}\leq 1/2$. Therefore, we choose $\th=\th(r,m)\in (0,1/3)$ such that $2C(r,m)\th^\be \leq \th^{\frac{\be}2}$ and $ \th^{\frac12-\frac1r}\leq 1/2$, then
\begin{align*}
W_{r,m}(v,q,h,f,\th^{n+1})   \leq \th^{(n+1)\al }\e_2
\end{align*}
for some $\al\in (0,1)$. 

Lastly, $(v)_{Q_{\th^{n+1}}}$ can be estimated as
\begin{align*}
|(v)_{Q_{\th^{n+1}}}|
&\leq \sum_{k=0}^n |(v)_{Q_{\th^{k+1}}} - (v)_{Q_{\th^k}}| + |(v)_{Q_1}|\\
&\leq \sum_{k=0}^n \left(\dint_{-\th^{2(k+1)}}
^0 \left(\dint_{B_{\th^{k+1}}} |v-(v)_{Q_{\th^k}}|^m \d x\right)^{\frac rm} \d t\right)^\frac1r + \frac{1}2\\
&\leq \th^{-(2-\si)} \sum_{k=0}^n W_{r,m}(v,q,h,f,\th^k) +\frac{1}2
\leq \th^{-(2-\si)}\sum_{k=0}^n \th^{k\al} \e_2 +\frac{1}2\\
&\leq \frac{\th^{-(2-\si)}}{1-\th^\al}\e_2 + \frac 12 \leq 1
\end{align*}
for the choice of sufficiently small $\e_2\in (0,\min(\e_1,1)]$.
This completes the proof of the claim.
\end{proof}

\medskip 
 
\noindent\texttt{Step 2.} H\"older regularity of $v$. 

By the translation and scale invariances of \eqref{per.NS}, a transformed solution
\begin{align*}
\td v (t,x )  = 2^{-1} v(2^{-2}t + t_0, 2^{-1} x + x_0)\\
\td q (t,x )  = 2^{-2} q(2^{-2}t + t_0, 2^{-1} x + x_0)\\
\td h (t,x )  = 2^{-1} h(2^{-2}t + t_0, 2^{-1} x + x_0)\\
\td f (t,x )  = 2^{-3} f(2^{-2}t + t_0, 2^{-1} x + x_0)
\end{align*} 
for $z_0:=(t_0,x_0)\in Q_{1/2}$ is also a suitable weak solution to \eqref{per.NS} in $Q_1$. Now, applying the claim to the transformed solution, we obtain that if $ |(v)_{Q_{1/2}(z_0)}| \leq 1$ and
\begin{align*}
\frac 12U_{r,m}(v,q, 2^{-1}, z_0) + \norm{h}_{L^r_tL^\infty_x(Q_{1/2}(z_0))}
+\norm{f}_{L^\frac r2_tL^\infty_x(Q_{1/2}(z_0))}
\le \e_2,
\end{align*}
then
\begin{align}\label{pre.holder}
\left( 
\dint_{\frac12 Q_{\th^k}(z_0)} |v-(v)_{\frac12Q_{\th^k}(z_0)}|^2 \d x \d t
\right)^\frac12 \leq 2\th^{k\al} \e_2, \qquad \forall k\in \mathbb{N}\cup\{0\}
\end{align}
where $\th$, $\e_2$, $\al$ are chosen as in the claim. Since we have
\begin{align*}
\sup_{z_0\in Q_{1/2}}|(v)_{Q_{1/2}(z_0)}|
\leq C_1 \norm{v}_{L^r_tL^m_x(Q_1)} \leq C_1\e_0
\end{align*}
and 
\begin{align*}
\sup_{z_0\in Q_{1/2}}&\left(\frac 12U_{r,m}(v,z, 2^{-1}, z_0) + \norm{h}_{L^r_tL^\infty_x(Q_{\frac12}(z_0))}
+\norm{f}_{L^\frac r2_tL^\infty_x(Q_{\frac12}(z_0))}\right)\\
&\leq C_1\left( \norm{v}_{L^r_tL^m_x(Q_1)} + \norm{q}_{L^\frac r2_tL^\frac m2_x(Q_1)}+
\norm{h}_{L^r_tL^\infty_x(Q_1)}
+\norm{f}_{L^\frac r2_tL^\infty_x(Q_1)}\right) \leq C_1\e_0
\end{align*}
for some universal constant $C_1$, we choose a positive constant $\e_0$ such that $C_1\e_0 \leq \min(\e_2,1)$ and apply Campanato's lemma to \eqref{pre.holder} to get
\begin{align*}
\norm{v}_{C^\al_{\text{par}}(Q_{1/2})} \leq C
\end{align*}
for some universal positive constant $C$. \end{proof}

\section{Proof of Theorem \ref{main.thm}}\label{sec:proof}

%In this section, we prove Theorem \ref{main.thm}.

\begin{proof}[Proof of Theorem \ref{main.thm}]

Applying Lemma \ref{lem:decomp}, we first have a decomposition of a dissipative solution on $Q_1$ as
\begin{align*}
u= v+h
\end{align*}
where $h$ is a harmonic function and divergence-free on $Q_1$ satisfying 
\begin{align*}
\norm{\na^k h}_{L^\infty((-t_0,0)\times B_1)}&\lec_k \norm{u}_{L^\infty(-t_0,0;L^1(B_2))}\\
\norm{h}_{L^r_tW^{1,\infty}_x(Q_1)} 
&\lec \norm{u}_{L^r_tL^m_x(Q_2)}
\end{align*}
for any $t_0\in (0,1]$.
For such $h$, $v$ is a suitable weak solution to \eqref{per.NS} in $Q_1$ with some pressure $q$, and some divergence-free external force $f$ satisfying
\begin{align*}
\norm{q}_{L^\frac r2_tL^\frac m2_x(Q_1)}
&\lec_m \norm{u}_{L^r_tL^m_x(Q_2)}^2\\
\norm{f}_{L^\frac r2_tL^\infty_x(Q_1)}
&\lec \norm{u}_{L^r_tL^m_x(Q_2)} + \norm{u}_{L^r_tL^m_x(Q_2)}^2. 
\end{align*}  
Therefore, using \eqref{est.v} we have
\begin{align*}
\norm{v}_{L^r_tL^m_x(Q_1)}
+\norm{q}_{L^\frac r2_tL^\frac m2_x(Q_1)}
+\norm{h}_{L^r_tW^{1,\infty}_x(Q_1)}
+\norm{f}_{L^\frac r2_tL^\infty_x(Q_1)}
\leq C_2 (\norm{u}_{L^r_tL^m_x(Q_2)}+\norm{u}_{L^r_tL^m_x(Q_2)}^2) 
\end{align*}
for some positive constant $C_2=C_2(m)$, and hence if we choose a positive constant $\e=\e(r,m)$ such that $C_2(m)(\e + \e^2)\leq \e_0(r,m)$, where $\e_0(r,m)$ is given by Theorem \ref{thm:ep.reg}, we have
\begin{align*}
\norm{v}_{C^{\al}_{\text{par}}(Q_{1/2})}\leq C 
\end{align*}
for some $\al\in (0,1/2-1/r)$ and some universal positive constant $C$.

\end{proof}

\section{Application: Short time local regularity of the Navier-Stokes equations} \label{sec:stlr}
In this section, we give a proof of Theorem \ref{thm:app}. To this end, we first introduce a Gr\"onwall-type inequality in \cite[Lemma 2.2]{BrTs20}.
\begin{lem}\label{Gronwall} Suppose $g\in L^\infty_\loc([0,T_0);[0,\infty))$ satisfies 
\begin{align*}
g(t) \leq a + b\int_0^t g(s) + g^m(s) \d s, \quad \forall t\in  (0,T_0)
\end{align*}
for some $a,b>0$ and $m\geq 1$. Then, we have $g(t)\leq 2a$ for $t\in (0,T)$ with 
\begin{align*}
T=  \min\left(T_0,\, \frac{\bar C}{b(1+a^{m-1})}\right) 
\end{align*}
where $\bar C$ is a universal constant.
\end{lem}

The proof of Theorem \ref{thm:app} is relying on the scheme in \cite{KMTs20}. 
\begin{proof}[Proof of Theorem \ref{thm:app}.]
First, we decompose the dissipative solution on $(0,T_0)\times B_1$ into the principle part $v$ and the harmonic part $h$ as in Lemma \ref{lem:decomp}---
the lemma is still valid with the replacement of 
the time interval $(-4,0)$ by $(0,T_0)$. Since $u\in L^\infty(0,T_0;L^2(B_2))$  and $h$ satisfies 
\begin{align*}
\norm{h}_{L^\infty(0,T_0;C^k(B_1))} \lec_k \norm{u}_{L^\infty(0,T_0;L^2(B_2))}
\end{align*}
by \eqref{est.h}, it is enough to show that $v$ is regular on $(0,T)\times B_{1/2}$. Also, we note that $v$ is a suitable weak solution to \eqref{per.NS} on $(0,T_0)\times B_1$ with initial data $v_0 := -\curl\De^{-1}(\ph \curl u_0)$, where $\ph$ is the spatial cut-off defined as in Lemma \ref{lem:decomp}. Indeed, using \eqref{est.v}, one can easily see that
\begin{align}\label{est.v0}
\norm{v_0}_{L^3(B_1)}
\lec \norm{u_0}_{L^3(B_2)} \lec  \e, \quad
\lim_{t\to0^+} \norm{v(t,\cdot)- v_0}_{L^2(B_1)} 
\lec \lim_{t\to0^+} \norm{u(t,\cdot)- u_0}_{L^2(B_1)}  =0.
\end{align}
By Theorem \ref{thm:ep.reg}, it then follows that for any $Q_r(t_0,x_0)$ contained in the domain $(0,T_0)\times B_1$, if we have
\begin{equation}\begin{split}
&r^{-\frac23}\norm{v}_{L^3(Q_r(t_0,x_0))}
+ r^{-\frac43}\norm{q}_{L^{\frac 32}(Q_r(t_0,x_0))} \\
&\hspace{3cm}+ r^\frac13\norm{h}_{L^3_tW^{1,\infty}_x(Q_r(t_0,x_0))}
+ r^{\frac53}\norm{f}_{L^\frac 32_tL^\infty_x(Q_r(t_0,x_0))} \le \e_0,
\end{split}\end{equation}
then $\norm{v}_{L^\infty(Q_{r/2}(t_0,x_0))}\leq C/r$, where $\e_0:=\e_0(3,3)$ and $C$ are universal constants defined as in the theorem. Since we have
\begin{equation}\label{est.hf}\begin{split}
&r^\frac13\norm{h}_{L^3_tW^{1,\infty}_x(Q_r(t_0,x_0))}
+ r^{\frac53}\norm{f}_{L^\frac 32_tL^\infty_x(Q_r(t_0,x_0))}\\
&\quad\leq r^\frac13\norm{h}_{L^3_tW^{1,\infty}_x((0,T_0)\times B_1)}
+ r^{\frac53}\norm{f}_{L^\frac 32_tL^\infty_x((0,T_0)\times B_1)}\\
&\quad\lec  C_0 r^\frac13\norm{u}_{L^3((0,T_0)\times B_2)}
+ C_0 r^{\frac53}(\norm{u}_{L^3((0,T_0)\times B_2)} + \norm{u}_{L^3((0,T_0)\times B_2)}^2)\\
&\quad \leq C_0 r^\frac13 M
+ C_0 r^{\frac53}(M + M^2),
\end{split}\end{equation}
where  $C_0$ is a universal constant obtained from \eqref{est.f} and \eqref{est.h}, we can get the desired smallness of $h$ and $f$ for sufficiently small $r$. Therefore, we aim to prove 
\begin{align}\label{smallness.vq}
\frac 1{r^2} \iint_{Q_r(t_0,x_0)} |v|^3 \d x \d t \leq \left(\frac{\e_0}{4}\right)^3, \quad
\frac 1{r^2} \iint_{Q_r(t_0,x_0)} |q|^\frac32 \d x \d t
\le \left(\frac{\e_0}{4}\right)^\frac32
\end{align}
for any $Q_r(t_0,x_0)\subset (0,T_0)\times B_1$ with $(t_0,x_0)\in (0,T)\times B_{1/2}$. 

To this end, we test the local energy inequality for $v$ (obtained from \eqref{LEI.v} with $\la=1$ with help of the $L^2$-continuity at $t=0$ in \eqref{est.v0}) with a spatial cut-off $\ph\in C_c^\infty(B_{2r};[0,1])$ satisfying $\ph =1$ on $B_r$ for $B_{2r}\subset B_1$, and divide it by $1/r$ to get
\begin{equation}\begin{split} \label{LEI.app}
&\frac 1{2r} \int_{B_r} |v(t)|^2 \,\d x 
+ \frac 1r \int_0^t \int_{B_r } |\na v|^2 \,\d x \d s\\
&\leq \frac 1{2r} \int_{B_{2r}} |v_0|^2 \,\d x 
+ \frac{C_1}{r^3} \int_0^t \int_{B_{2r}} |v|^2  \,\d x \d s
+ \frac{C_1}{r^2} \int_0^t \int_{B_{2r}} |v|^3 + |h|^3 + |q|^\frac32  \,\d x \d s\\
&\quad+C_1r\int_0^t \int_{B_{2r}} |\na h|^3  \,\d x \d s + \frac{C_1}{r^\frac12} \int_0^t\int_{B_{2r}} |f|^\frac32  \,\d x \d s,
\end{split}\end{equation}
for any $t\in (0,T_0)$ for some universal constant $C_1$, where we used the Young's inequalities. Set 
\begin{align*}
E_{r}(t)
:= \frac 1{2r} \int_{B_r} |v(t)|^2 \,\d x 
&+ \frac 1r \int_0^t \int_{B_r } |\na v|^2 \,\d x \d s + \frac 1{r^2} \int_0^t \int_{B_r} |q|^\frac32 \,\d x \d s, \\
&\mathcal{E}_{r_1,r_2}(t)
:=\sup_{r_1\le r\le r_2} E_r(t)
\end{align*} 
and we will get the Gronwall-type inequality in Lemma \ref{Gronwall} for $\mathcal{E}_{r,\rho_0}(t)$ for some $\rho_0$. 

To estimate the right hand side of \eqref{LEI.app}, we let $\rho := c_0 r> 2r$. %and consider $r$ small enough to have $\rho\in (0,3/8)$. 
The constant $c_0>2$ will be determined below. First, we decompose $q$ into $q_o$ and $q_h$, where
$q_o = (-\De)^{-1} \div\div((v\otimes v + v\otimes h + h\otimes v)\chi_{\rho} ) $
for the indicator function $\chi_{\rho}$ of the set $B_{\rho}\subset B_1$. By the definition and the $L^p$-norm preservation of Riesz transforms, we get
\begin{align*}
\int_0^t\int_{B_\rho} |q_o|^\frac32  \, \d x \d s
\lec \int_0^t \int_{B_\rho} |v|^3 + |h|^3 \, \d x \d s.
\end{align*}
Then, since $q_h:= q-q_o$ is harmonic on $(0,T_0)\times B_\rho$, using elliptic estimates we have
\begin{align}\label{est.qh32}
\int_0^t \int_{B_{2r}} |q_h|^\frac 32 \, \d x \d s
&\lec \frac{r^3}{\rho^3} \int_0^t \int_{B_{\rho}} |q_h|^\frac32 \, \d x \d s
\lec \int_0^t \int_{B_{\rho}} |q_o|^\frac32 \, \d x \d s
+\frac{r^3}{\rho^3} \int_0^t \int_{B_{\rho}} |q|^\frac32 \, \d x \d s.
\end{align}
Combining two estimates, we have
\begin{align}\label{est.q32}
\frac {C_1}{r^2} \int_0^t \int_{B_{2r}} |q|^\frac32 \, \d x \d s
\leq \frac {C_2}{r^2}  \int_0^t \int_{B_\rho} |v|^3 + |h|^3 \, \d x \d s  + \frac{ C_2 r}{\rho}  E_{\rho}(t)
\end{align}
for some universal constant $C_2$, and choose $c_0 = 8C_2$ to have $(C_2r)/\rho = 1/8$. (If necessary, we adjust $C_2$ in \eqref{est.q32} to have $c_0>2$.) On the other hand, by the standard interpolations, we have
\begin{align}
\label{est.v3}
&\frac{(C_1+C_2)}{r^2} \int_0^t \int_{B_{\rho}} |v|^3\,\d x \d s \nonumber\\
&\hspace{2cm}\lec \frac{1}{r^2} \int_0^t \norm{v(s,\cdot)}_{L^2(B_{\rho})}^\frac32 \norm{\na v(s,\cdot)}_{L^2(B_{\rho})}^\frac32+r^{-\frac32}\norm{ v(s,\cdot)}_{L^2(B_{\rho})}^3 \,\d s\\
&\hspace{2cm}\leq  \frac 1{16\rho} \int_0^t\int_{B_\rho} |\na v|^2 \,\d x \d s 
+ c \int_0^t \frac {\rho^3}{r^8}\norm{v(s,\cdot)}_{L^2(B_{\rho})}^6+\frac {1}{r^{\frac72}} \norm{ v(s,\cdot)}_{L^2(B_{\rho})}^3 \d s \nonumber\\
&\hspace{2cm}\leq \frac 1{16} E_{\rho}(t)
+ \frac{C_3}{r^2} \int_0^t E_{\rho}(s)^3 + E_{\rho}(s)^\frac32 \,\d s.\nonumber
\end{align}
for some universal constants $c$ and $C_3$.
Using this, \eqref{est.v0}, \eqref{est.q32}, and \eqref{est.hf}, and taking supremum in $[r,r_0]$ to \eqref{LEI.app}, we first obtain
\begin{equation}\label{est.mathcalE1}\begin{split}
\mathcal{E}_{r,r_0}(t) 
&\leq C_4\e^2 
+ \frac {C_4}{r^2} \int_0^t  \mathcal{E}_{r,\rho_0}(s)+ \mathcal{E}_{r,\rho_0}(s)^3  \d s 
+\frac 14 \mathcal{E}_{r,\rho_0}(t) + C_4 \rho_0 M^3
+ C_4 \rho_0^{\frac52} (M^\frac32+ M^3)
\end{split}\end{equation}
for some universal constant $C_4$, where $\rho_0 := c_0 r_0$ and $r_0$ will be chosen below sufficiently small so as to $\rho_0<1/8$. 

We now estimate the remaining piece $\mathcal{E}_{r_0, \rho_0}(t)$. By \eqref{LEI.app}, we have
\begin{equation}\begin{split}\label{est.Er2}
E_r(t) 
&\lec \norm{v_0}_{L^3(B_{3/4})}^2
+ \frac {t}{r^3} \norm{v}_{L^\infty(0,t;L^2(B_{3/4}))}^2 
+ \frac {t^\frac 14}{r^2} \norm{v}_{L^4(0,t;L^3(B_{3/4}))}^3
+ \frac{1}{r^2} \norm{q}_{L^\frac32((0,t)\times B_{2r})}^\frac32\\
&\quad + r\norm{h}_{L^3(0,t;W^{1,\infty}(B_1))}^3 
+ r^\frac52 \norm{f}_{L^\frac32(0,t;L^\infty(B_1))}^\frac32
\end{split}\end{equation}
for $r\in [r_0, \rho_0]$.  
To estimate the energy norm of $v$ on $(0,T_0)\times B_{3/4}$, we again test the local energy inequality for $v$ with a spatial smooth cut-off $\ac\ph\in C_c^\infty(B_1;[0,1])$ with $\ac\ph=1$ on $B_{3/4}$ and use H\"older and Young's inequalities, \eqref{est.v}, and \eqref{est.q}-\eqref{est.h} to get
\begin{align}
&\norm{v}_{L^\infty(0,T_0;L^2(B_{3/4}))}^2
+\norm{\na v}_{L^2((0,T_0)\times B_{3/4})}^2\nonumber\\
&\quad\lec \norm{v_0}_{L^2(B_1)}^2+\norm{v}_{L^2((0,T_0)\times B_1)}^2
+\norm{v}_{L^3((0,T_0)\times B_1)}^3 \nonumber\\
&\qquad+\norm{h}_{L^3(0,T_0;W^{1,3}( B_1))}^3
+\norm{q}_{L^\frac32((0,T_0)\times B_1)}^\frac32
+\norm{f}_{L^\frac32((0,T_0)\times B_1)}^\frac32 \nonumber\\
&\quad\lec \e^2+ \norm{u}_{L^3((0,T_0)\times B_2)}^\frac32
+ \norm{u}_{L^3((0,T_0)\times B_2)}^3
\lec\e^2 + M^\frac32 + M^3\label{est.energy.v}.
\end{align}
Then, setting $\hat{q}_o:= (-\De)^{-1} \div\div ((v\otimes v+ v\otimes h+ h\otimes v)\chi_{3/4})$ and $\hat{q}_h := q-\hat{q}_o$, we estimate $q$ as before; for any $r\in [r_0, \rho_0]$ and $t\in [0,T_0]$,
\begin{align*}
\norm{q}_{L^\frac32((0,t)\times B_{2r})}^\frac32
&\lec
\norm{\hat {q}_o}_{L^\frac32((0,t)\times \R^3)}^\frac32
+\norm{\hat{q}_h}_{L^\frac32((0,t)\times B_{2r})}^\frac32
\lec
\norm{\hat {q}_o}_{L^\frac32((0,t)\times \R^3)}^\frac32
+r^3\norm{\hat{q}_h}_{L^\frac32((0,t)\times B_{3/4})}^\frac32 \\
&\lec \norm{\hat{q}_o}_{L^\frac32((0,t)\times \R^3)}^\frac32
+ r^3 \norm{q}_{L^\frac32((0,t)\times B_{3/4})}^\frac32\\
&\lec t^\frac 14\norm{v}_{L^4(0,t;L^3( B_{3/4}))}^3
+t^\frac12\norm{v}_{L^\infty(0,t;L^2( B_{3/4}))}^\frac32\norm{h}_{L^3(0,t;L^\infty(B_{3/4}))}^\frac32
+ r^3 \norm{q}_{L^\frac32((0,t)\times B_{3/4})}^\frac32\\
&\lec t^\frac 14(\e^3 + M^\frac94+ M^\frac92)
+ t^\frac 34\norm{h}_{L^3(0,t;L^\infty(B_{3/4}))}^3
+ r^3 \norm{q}_{L^\frac32((0,t)\times B_{3/4})}^\frac32,
\end{align*}
where %the second line is obtained similar to \eqref{est.qh32} and the fourth inequality follows from the $L^p$-norm preservation of Riesz transforms. 
the last line follows from the Young's inequality and \eqref{est.energy.v}. Using these, \eqref{est.hf}, and \eqref{est.q}, and taking supremum in $[r_0,\rho_0]$ to \eqref{est.Er2}, we get
\begin{align*}
\mathcal{E}_{r_0, \rho_0}(t)
&\lec \e^2
+ \frac {t} {r_0^3} (\e^2+ M^\frac32+ M^3)
+ \frac{t^\frac14}{r_0^2}(\e^3+ M^\frac94+ M^\frac92)
+  \frac{t^\frac 34}{r_0^2} M^3\\
&\quad+ \rho_0 M^3 +  \rho_0^\frac52 (M^\frac 32+M^3).
\end{align*}

Combining it with \eqref{est.mathcalE1}, we finally have
\begin{equation}\begin{split}
%\label{pre.gronwall}
\mathcal{E}_{r, \rho_0}(t)
&\lec  \frac {1}{r^2} \int_0^t  \mathcal{E}_{r,\rho_0}(s)+ \mathcal{E}_{r,\rho_0}(s)^3  \d s + \e^2 + \frac {\bar T} {r_0^3} (\e^2+M^\frac32+ M^3) \\
 &\quad
+ \frac{ \bar T^\frac14}{r_0^2}(\e^3+ M^\frac94+ M^\frac92)
+  \frac{ \bar T^\frac 34}{r_0^2} M^3
+\rho_0 M^3 +  \rho_0^\frac52 (M^\frac 32+M^3)\\
&\leq \frac {C_5}{r^2} \int_0^t  \mathcal{E}_{r,\rho_0}(s)+ \mathcal{E}_{r,\rho_0}(s)^3  \d s + C_5(\e^2+\de^\frac14).
\end{split}\end{equation}
for $t\in  [0,\bar T]$, $\bar T\in (0,T_0]$, and $r\leq r_0$ and for some universal constant $C_5\geq 1$. The last line follows from the choice of $\rho_0 =\de/(1+M^3)$ (and hence $r_0 = \de/(c_0(1+M^3))$) and $\bar T=\min(T_0,{\de^9}/{(1+M^{42})})$ and the assumption $\e\in(0,1)$.  
Here, $\de\in(0,1)$ will be chosen below as a universal constant. Lastly, we use Lemma \ref{Gronwall} to have
\begin{align*}
\mathcal{E}_{r, \rho_0}(t)
\leq 2C_5 (\de^\frac{1}4 + \e^2), \quad \forall t\in (0,\min(\bar T, c_1r^2)),\,  r\in (0, r_0]
\end{align*}
where $c_1=\bar C C_5^{-3}(1+\e^2+ \de^\frac14)^{-2}$ for the universal constant $\bar C$ in Lemma \ref{Gronwall}.
(If necessary, we adjust $C_5$ to have $c_1\le 1$.) 
It then implies that similar to \eqref{est.v3}, 
\begin{align*}
\frac 1{(\sqrt{c_1}r)^2} \int_{Q_{r\sqrt{c_1}}(c_1r^2, x_0)} |v|^3 \d x\d s 
&\leq 
\frac 1{(\sqrt{c_1}r)^2} \int_0^{c_1r^2}\int_{B_r(x_0)} |v|^3 \d x\d s\\
&\lec
\frac 1{(\sqrt{c_1}r)^2} \int_0^{c_1r^2}
\norm{v}_{L^2(B_r(x_0))}^\frac32\norm{\na v}_{L^2(B_r(x_0))}^\frac32 + r^{-\frac32}\norm{v}_{L^2(B_r(x_0))}^3 \d s\\
&\lec 
(c_1^{-\frac34} + 1) \sup_{t\in [0,c_1r^2]} E_r^\frac32(t)
\lec (\de^\frac18 + \e)^3
\end{align*}
and 
\begin{align*}
\frac 1{(\sqrt{c_1}r)^2}\int_{Q_{r\sqrt{c_1}}(c_1r^2, x_0)} |q|^\frac32 \d x\d s 
\lec c_1^{-1} E_r(c_1r^2) \lec (\de^\frac18 + \e)^2.
\end{align*}
for any $x_0\in B_{1/2}$, provided that $r\leq \min(r_0, \sqrt{c_1^{-1}\bar T})$. Here, we used $c_1^{-1}\lec (1+\e^2+\de^\frac14)^2\lec 1$ and the generalization to $x_0\in B_{1/2}$ easily follows by repeating the same argument.  

Finally, we choose $\de,\e \le(0,1)$ such that \eqref{smallness.vq} holds on $Q_{r\sqrt{c_1}}(c_1r^2, x_0)$, $\varrho_0\leq \de <1/8$, and 
\begin{align*}
&(\sqrt{c_1}r)^{\frac13}\norm{h}_{L^3_tW^{1,\infty}_x(Q_{r\sqrt{c_1}}(c_1r^2, x_0))} + (\sqrt{c_1}r)^{\frac53}\norm{f}_{L^\frac32_tL^\infty_x(Q_{r\sqrt{c_1}}(c_1r^2, x_0))}\\
&\hspace{3cm}\leq C_0(\sqrt{c_1}r_0)^{\frac13} M + C_0 (\sqrt{c_1}r_0)^{\frac53}(M+M^2)
\lec \de^\frac13 \leq \frac {\e_0}2
\end{align*}
for $r\leq r_0$. Then, for almost every $(t_0,x_0)\in (0, \min(c_1r_0^2, \bar T))\times B_{1/2}$, we have
\begin{align*}
|v(t_0,x_0)|\leq \frac{ C}{\sqrt{t_0}}. 
\end{align*}
We complete the proof setting $T :=\min(T_0, C_*(1+M^{42})^{-1}) \leq \min(c_1r_0^2, \bar T)$ for some universal constant $C^*>0$. 

\end{proof}

\section{Local suitable weak solutions and dissipative solutions}\label{sec:appenB}
In this section, we prove Theorem \ref{lem:equi}. We first remark that since $u$ is divergence-free, $p_{h, B_R}$ is harmonic and satisfies
\begin{align*}
\norm{\na p_{h, B_R}(t,\cdot)}_{L^2(B_R)} 
\lec \norm{u(t,\cdot)}_{L^2(B_R)}, \quad \forall \text{a.e. }t\in (a,b).
\end{align*}

\begin{proof}[Proof of Theorem \ref{lem:equi}.] 
For the convenience, we drop $B_R$ in the index of $p_{h,B_R}$ and $p_{o,B_R}$. Let $u$ be a local suitable weak solution. By \cite[Lemma 2.4]{Wolf15}, any weak solutions to the Navier-Stokes equations solves
\begin{align}\label{eqn.u0}
\pa_t u  + \div (u\otimes u ) + \na (p_o + \pa_t p_h) = \De u, \quad\text{ in }\mathcal{D'},
\end{align}
which is equivalent to, setting $v:= u+\na p_h$,  
\begin{align}\label{eqn.v}
\pa_t v  + \div (v\otimes u ) - \div( \na p_h\otimes u) + \na p_o = \De v
\quad\text{ in }\mathcal{D'}.
\end{align}
We now show that $(u,p)$ for $p=p_o + \pa_t p_h$ is a dissipative solution. Similar to \texttt{Step 2} in the proof of Lemma \ref{lem:decomp}, we fix a non-negative test function $\xi \in C_c^\infty((0,T)\times B_R)$ and test \eqref{eqn.u0} and \eqref{eqn.v} with $u_{\al,\ep}\xi \ast \psi_{\al,\ep}$ and $v_{\al,\ep}\xi \ast \psi_{\al,\ep}$ to get
\begin{align*}
&\left((\pa_t-\De)  \frac{|u_{\al,\ep}|^2}2+|\na u_{\al,\ep}|^2
+ u_{\al,\ep}\cdot (\div (u\otimes u)\ast\psi_{\al,\ep}) + \div((p\ast \psi_{\al,\ep}) u_{\al,\ep})\right)(\xi)  =0\\
&\left((\pa_t-\De)  \frac{|v_{\al,\ep}|^2}2+|\na v_{\al,\ep}|^2
+ v_{\al,\ep}\cdot (\div (v \otimes u)\ast\psi_{\al,\ep})\right)(\xi)\\
&\hspace{3.7cm}+\left(- v_{\al,\ep}\cdot (\div (\na p_h \otimes u)\ast\psi_{\al,\ep}) + \div((p_o\ast \psi_{\al,\ep}) v_{\al,\ep})\right)(\xi)  =0,
\end{align*}
and send $\al\to 0$ and $\ep\to 0$ to get
\begin{align*}
&\left((\pa_t-\De)  \frac{|u|^2}2+|\na u|^2
+ \div \left(\frac{|u|^2u}2\right) + \lim_{\ep\to 0} \lim_{\al\to 0}\div((p\ast \psi_{\al,\ep}) u_{\al,\ep})\right)(\xi)  = -\lim_{\ep\to 0} \mu_\ep (\xi)\\
&\left((\pa_t-\De)  \frac{|v|^2}2+|\na v|^2
+ \div \left(\frac{|v|^2u}2\right)- v\cdot \div (\na p_h \otimes u) + \div(p_o v)\right)(\xi)  = - \lim_{\ep \to 0} \td \eta_\ep (\xi).
\end{align*}
Here, $\psi_{\al,\ep}$ is defined as in Definition \ref{def:diss} for sufficiently small $\al$ and $\ep$ and $u_{\al,\ep}, v_{\al,\ep}, u_\ep, v_\ep$ are defined as in the proof of Lemma \ref{lem:decomp}. Also, we write $\mu_\ep := u_{\ep}\cdot (\div (u\otimes u)\ast\ga_{\ep}) - \div (|u|^2 u/2)$, $\td \eta_\ep := v_{\ep}\cdot (\div (v \otimes u)\ast\ga_{\ep}) - \div (|v|^2u/2)$. 
Note that we have $\underset{\ep\to 0}{\lim} (\mu_\ep - \td \eta_\ep + T_\ep(u,v,v) - T_\ep(u,u,u)) =0$, where
\begin{align*}
T_\ep(U,V,W)
: =& -\frac12 \int \tau_y[U](t,x)\cdot\na\ga_\ep(y) (\tau_y[V](t,x)\cdot \tau_y[W](t,x)) \d y \\
&+ \int \tau_y[U](t,x) \cdot\na \ga_\ep(y) (\tau_y[V](t,x)\cdot (W\ast \ga_\ep-W)(t,x) \d y
\end{align*}
and $\tau_z(Y)(t,x):=Y(t,x-z)-Y(t,x)$ (see \cite[Lemma 3.7]{CLRM18} for the justification), Then, we have $\underset{\ep\to 0}{\lim} (\mu_\ep - \td\eta_\ep)(\xi) =\underset{\ep\to 0}{\lim}(T_\ep(u,u,u) -T_\ep(u,v,v))(\xi)  =0$ in $ \mathcal{D}'$ because of \cite[Lemma 3.8]{CLRM18} and {$\na p_h \in L^\infty_t\text{Lip}_x(\supp(\xi))$.} As a result,  $\langle \div(pv)\rangle(\xi)$ is well-defined independent of the mollification. Since $u$ is a local suitable weak solution, $\underset{\ep\to 0}{\lim}\,\td \eta_\ep(\xi)$ exists and is non-negative, and so does $\underset{\ep\to 0}{\lim}\,\mu_\ep(\xi)$.  Therefore, $(u,p)$ is a dissipative solution.

On the other hand, if $(u,p)$ is a dissipative solution, again by \cite[Lemma 2.4]{Wolf15}, we have the representation $\na p = \na(p_o + \pa_t p_h)$ in $\mathcal{D}'$ and $(u,p)$ satisfies \eqref{eqn.u0} and \eqref{eqn.v}. Since we now have $\underset{\ep\to 0}{\lim} \underset{\al\to 0}{\lim}\div((p\ast \psi_{\al,\ep}) u_{\al,\ep})(\xi)=\langle \div(pv)\rangle(\xi)$ and $\underset{\ep\to 0}{\lim}\,\mu_\ep(\xi)$ is well-defined and non-negative, it follows that $\underset{\ep\to 0}{\lim}\,\td \eta_\ep(\xi)$ exists and non-negative. Therefore, $u$ is a local suitable weak solution. 
\end{proof}

\appendix
\section{Heat semigroup estimates}
In this section, we recall several heat semigroup estimates. For the proof, see \cite[Appendix D]{JJW16} and \cite[Proposition 11]{BP20}. 
\begin{lem}\label{lem:semigp}		
For any $1\le p<q<\infty$ with $\frac1q +\frac{1}{5} = \frac1p$, we have
\begin{align*}
\Norm{\int_{-1}^t e^{(t-\tau)\De}\div g(\tau)\d \tau}_{L^q([-1,0]\times \R^3)} \lec_{p,q} \norm{g}_{L^p([-1,0]\times \R^3)}.
\end{align*}

\end{lem}

\begin{lem}\label{lem:semigp2} For any $r, p\in [1,\infty]$ satisfying $\frac2r + \frac3p<2$, we have
\begin{align*}
\Norm{\int_{-1}^t e^{(t-\tau)\De} f(\tau)\d \tau}_{C^\ga_{\text{par}}([-1,0]\times \R^3)} 
\lec_{r,p} \norm{f}_{L^r([-1,0];L^p( \R^3))}
\end{align*}
for all $\ga \in (0, \min(1-\frac 1r-\frac 3{2p}, \frac12))$. Also, for any $r, p\in [1,\infty]$ satisfying $\frac2r + \frac3p<1$, we have
\begin{align*}
\Norm{\int_{-1}^t e^{(t-\tau)\De} \div g(\tau)\d \tau}_{C^\ga_{\text{par}}([-1,0]\times \R^3)} 
\lec_{r,p} \norm{g}_{L^r([-1,0];L^p( \R^3))}
\end{align*}
for all $\ga\in (0, \frac12-\frac1r-\frac3{2p})$ .
\end{lem}

%\bibliographystyle{abbrv}
%\bibliography{Kwon}

\end{document}